\documentclass[10pt,leqno]{amsart}
\usepackage{txfonts}
\usepackage{amsmath,amsthm,amsfonts}
\usepackage{latexsym}
\usepackage{amssymb}
\usepackage{mathrsfs}
\usepackage[colorlinks,
            linkcolor=red,
            anchorcolor=blue,
            citecolor=green,
            ]{hyperref}
\usepackage[dvips]{epsfig}

\newcounter{RomanNumber}

\renewcommand{\thesection}{\arabic{section}}

\newtheorem{theorem}{Theorem}[]
\newtheorem{lemma}{Lemma}[section]
\newtheorem{prop}{Proposition}[section]

\theoremstyle{remark}
\newtheorem{remark}{Remark}[section]

\renewcommand{\theequation}{\thesection .\arabic{equation}}
\let\subs\subsection
\renewcommand\subsection{\setcounter{equation}{0}
\gdef\theequation{\thesubsection \arabic{equation}}\subs}
\let\sect\section
\renewcommand\section{\setcounter{equation}{0}
\gdef\theequation{\thesection .\arabic{equation}}\sect}

\makeatletter

\newcommand{\Rmnum}[1]{\expandafter\@slowromancap\romannumeral #1@}
\makeatother

\newcommand{\IC}{{\mathbb{C}}}

\newcommand{\IR}{{\mathbb{R}}}

\newcommand{\be}{\begin{eqnarray}}
\newcommand{\ee}{\end{eqnarray}}

\newcommand{\mes}{\mathop{\rm{mes}\, }}

\def\beeq{\begin{equation}}
\def\eneq{\end{equation}}

\def\bm{\begin{matrix}}
\def\endm{\end{matrix}}

\def\Im{{\rm Im}}
\def\Re{{\rm Re}}

\begin{document}

\title[Strong Birkhoff Ergodic Theorem  and  its application]{Strong Birkhoff Ergodic Theorem for subharmonic functions with irrational shift and its application to analytic quasi-periodic cocycles}

\author{Kai Tao}
	\address{College of Sciences, Hohai University, 1 Xikang Road Nanjing Jiangsu 210098 P.R.China}
	\email{ktao@hhu.edu.cn,\ tao.nju@gmail.com}

\thanks{The first author was supported by  the National Nature Science Foundation of China (Grant 11401166).  }

\date{}

\begin{abstract}
In this paper, we first prove  the strong Birkhoff Ergodic Theorem for subharmonic functions with the irrational shift on the Torus. Then, it is applied   to the analytic quasi-periodic  Jacobi cocycles. We show that if the Lyapunov exponent of these cocycles is positive at one point, then it is positive on an interval centered at this point for suitable frequency and coupling numbers.  We also  prove that the Lyapunov exponent is H\"older continuous in $E$ on this interval and calculate the expression of its length. What's more, if the coupling number of the potential is large, then the Lyapunov exponent is always positive for all irrational frequencies and H\"older continuous in $E$ for all finite Liouville frequencies. We also study the Lyapunov exponent of  the Schr\"odinger cocycles, a special case of the Jacobi ones, and obtain its H\"older continuity  in the frequency.
\end{abstract}

\maketitle

\section{Introduction}
By the Birkhoff Ergodic Theorem, if $T:X \to X$ is an ergodic transformation on a measurable space $(X,\Sigma,m)$  and $f$ is an $m-$integrable  function, then the time average functions $f_n(x)=\frac{1}{n}\sum_{k=0}^{n-1}f(T^kx)$ converge to the space average $<f>=\frac{1}{m(X)}\int_X fdm$ for almost every $x\in X$. But it doesn't tell us how fast do they converge? So, we call a theorem the strong Birkhoff Ergodic Theorem, if it gives the convergence rate.

In this paper, we consider the strong Birkhoff Ergodic Theorem for subharmonic functions under the condition that the ergodic transformation is a shift on the Torus, i.e. $T:x\to x+\omega, \forall x \in \mathbb{T}:=[0,1]$. Specifically,  assume that  $u:\Omega\to \IR$ is a subharmonic function on
a domain $\Omega\subset\IC$, $\partial \Omega$ consists
of finitely many piece-wise $C^1$ curves and $\mathbb{T}\varsubsetneq \Omega$. Then, the Reisz Decomposition Theorem  tells us that  there exists a positive
measure $\mu$ on~$\Omega$ such that for any $\Omega_1\Subset \Omega$
(i.e., $\Omega_1$ is a compactly contained subregion of~$\Omega$),
\begin{equation}
\label{eq:rieszrep1} u(z) = \int_{\Omega_1}
\log|z-\zeta|\,d\mu(\zeta) + h(z),
\end{equation}
where $h$ is harmonic on~$\Omega_1$ and $\mu$ is unique with this
property.

In order to formulate our theorem, some notations about the shift $\omega$, which is always irrational in this paper, should be introduced. For any irrational $\omega$, there exist its continued fraction approximants $\{\frac{p_s}{q_s}\}_{s=1}^{\infty}$,  satisfying
\begin{equation}\label{irtor0}
 \frac{1}{q_s(q_{s+1}+q_s)}<|\omega-\frac{ p_s}{q_s}|<\frac{1}{q_sq_{s+1}}.
\end{equation}
Define $\bar\beta$ as the exponential growth exponent of $\{\frac{p_s}{q_s}\}_{s=1}^{\infty}$ as follows:
\[
 \bar\beta(\omega):=\limsup_s \frac{\log q_{s+1}}{q_s}\in[0,\infty].
\]
Thus, if $\omega$ is a finite Liouville frequency, which means $\bar\beta(\omega)<\infty$, then for any $\kappa>0$, there exists $s_0=s_0(\omega,\kappa)\ge 0$ such that for any $s \ge s_0$, $\log q_{s+1}\leq (\bar\beta+\kappa)q_s$. Therefore,   there exists a constant $ \beta(\omega)<\infty$ such that for any $s\ge 0$,
\[ \log q_{s+1}\leq \beta q_s.\]
Then, our strong Birkhoff Ergodic Theorem for  irrational $\omega$ is as follows:
\begin{theorem}\label{erg}
Let $u:\Omega\to \IR$ be a subharmonic function on
a domain $\Omega\subset\IC$. Suppose that $\partial \Omega$ consists
of finitely many piece-wise $C^1$ curves and  $\mathbb{T}$ is contained in $\Omega_1\Subset \Omega$. There exist a constant $C=C(\Omega_1)$ and an absolute constant $c$ such that for any $\delta>0$ and irrational $\omega$, if $ \beta(\omega)<\frac{\delta}{C\mu(\Omega_1)}$, then for any positive $n$,
\begin{equation}\label{diave}
  \mes\left (\left \{x\in \mathbb{T}:|\frac{1}{n}\sum_{k=1}^n u(x+k\omega)-<u>|>\delta \right \}\right )<\exp(-\frac{c}{\mu(\Omega_1)}\delta n),
\end{equation}where $\mu$ is the unique measure defined in (\ref{eq:rieszrep1}).
\end{theorem}
\begin{remark}\label{ergforbeta}
It is obvious that $\beta(\omega)\gtrsim\bar \beta(\omega)$. If we replace the assumption $\beta(\omega)<\frac{\delta}{C\mu(\Omega_1)}$ by $\bar\beta(\omega)<\frac{\delta}{C\mu(\Omega_1)}-\kappa$, then Theorem \ref{erg}  still  holds. But the absolute constant $c$ will depend on $\omega$. See details in Remark \ref{barbeta}.
\end{remark}
\begin{remark}\label{optimal}
In \cite{GS}, Goldstein and Schlag proved that for any strong Diophantine $\omega$, which satisfies the strong Diophantine condition
\begin{equation}\label{13001}
\|n\omega\| \geq \frac{C_{\omega}}{n(\log n)^\alpha}\ \ \mbox{for
all}\ n\not=0,\end{equation}(\ref{diave}) holds when $\delta>\frac{(\log n)^{\alpha+2}}{n}$.Obviously, this $\omega$ satisfies $\bar\beta(\omega)=0$. Thus, we extend their conclusion to the finite Liouville frequency. What's more, \cite{ALSZ} shows that our result is also optimal, as (\ref{diave}) can not hold when $\bar\beta(\omega)=\infty$.
\end{remark}
~\\

In this paper, we apply Theorem \ref{erg} to the following quasi-periodic analytic Jacobi operators $H_{x,\omega}$ on
$l^2(\mathbb{Z})$:
\begin{equation}\label{jacobiequ}
(H_{x,\omega}\phi)(n)=-\lambda_aa(x+(n+1)\omega)\phi(n+1)-\lambda_a\overline{a(x+n\omega)}\phi(n-1)+\lambda_vv(x+n\omega)\phi(n)
,\ n\in\mathbb{Z},\end{equation}where $v:\mathbb{T}\to\mathbb{R}$ is a real analytic function called potential, $a:\mathbb{T}\to\mathbb{C}$ is a complex analytic function and not identically zero, and $\lambda_a$ and $\lambda_v$ are real positive constants  called  coupling numbers. Then, their characteristic equations  $H_{x,\omega}\phi=E\phi$ can be expressed as
\begin{equation}
\left (\begin{array}{cc}
  \phi(n+1) \\ \phi(n) \\
\end{array} \right )=\frac{1}{\lambda_aa(x+(n+1)\omega)}\left ( \begin{array}{cc}
 \lambda_v v(x+n\omega)-E & -\lambda_a\overline{a(x+n\omega)} \\
 \lambda_aa(x+(n+1)\omega)& 0 \\
  \end{array}\right )\left (\begin{array}{cc}
  \phi(n) \\ \phi(n-1) \\
\end{array} \right ).
\end{equation}Define
\begin{equation}\label{eq:1.3}
  M(x,E,\omega):=\frac{1}{\lambda_aa(x+\omega)}\left ( \begin{array}{cc}
 \lambda_v v(x)-E & -\lambda_a\overline{a(x)} \\
 \lambda_aa(x+\omega)& 0 \\
  \end{array}\right )
\end{equation}and  call a map
$$(\omega,M):(x,\vec v)\mapsto (x+\omega,M(x)\vec v)$$ a Jacobi cocycle. Due to the fact that an analytic function only has  finite zeros, $M(x,E,\omega)$ and  the n-step transfer matrix $M_n(x,E,\omega):=\prod_{k=n}^1M(x+k\omega,E)$ make sense almost everywhere. By the Kingman's subadditive ergodic theorem, the Lyapunov exponent
\begin{equation}\label{leomega}
  L(E,\omega)=\lim\limits_{n\to\infty}L_n(E,\omega)=\inf\limits_{n\to\infty}L_n(E,\omega)
\end{equation} always exists, where
\[L_n(E,\omega)=\frac{1}{n}\int\limits_{\mathbb{T}}\log\|M_n(x,E,\omega)\|dx.\]
To apply Theorem \ref{erg} and the Avalanche Principle(Proposition \ref{prop:AP}),  we need to consider the following two matrices associated with $M_n$:
\[
  M_n^a(x,E,\omega)=\left(\prod_{j=1}^n\lambda_aa(x+j\omega)\right)M_n(x,E,\omega)\
\mbox{
and}\
  M^u_n(x,E,\omega)=\frac{M_n(x,E,\omega)}{|\det
M_n(x,E,\omega)|^{\frac{1}{2}}}.
\]
Note that an analytic function $f(x)$ on $\mathbb{T}$ has its complex analytic extension $f(z)$ on the complex strip $\mathbb{T}_{\rho}=\{z:|\Im z|<\rho\}$ and the complex analytic extension of $\bar a(x)$ should be defined on  $\mathbb{T}_{\rho}$ by
\[\tilde a(z):=\overline{a(\frac{1}{z})}.\]
Then, the extensions of $M(x,E,\omega)$, $M_n(x,E,\omega)$, $M_n^a(x,E,\omega)$ and $M_n^u(z,\omega,E)$ are
\begin{eqnarray}
 M(z,E,\omega)&=&\frac{1}{\lambda_aa(z+\omega)}\left ( \begin{array}{cc}
 \lambda_v v(z)-E & -\lambda_a\tilde a(z) \\
 \lambda_aa(z+\omega)& 0 \\
  \end{array}\right ), \  M_n(z,E,\omega)=\prod_{k=n}^1M(z+k\omega,E), \nonumber\\
   M_n^a(z,E,\omega)&=&\left(\prod_{j=1}^n\lambda_aa(z+j\omega)\right)M_n(z,E,\omega) \ \mbox{and} \ M^u_n(z,E,\omega)=\frac{M_n(z,E,\omega)}{|\det
M_n(z,E,\omega)|^{\frac{1}{2}}}.\label{de2}
\end{eqnarray}With fixed $\omega$ and $E$,
\begin{equation}\label{uan}
  u^a_n(z,E,\omega)=\frac{1}{n}\log\|M^a_n(z,E,\omega)\|
\end{equation} is a subharmonic function on $\mathbb{T}_{\rho}$ such that Theorem \ref{erg} can be applied. We also consider the quantities $u_n(z,E,\omega)$, $u_n^u(z,E,\omega)$, $L^a(E,\omega),\ L_n^a(E,\omega),\ L^u(E,\omega)$ and $L_n^u(E,\omega)$ which are defined analogously. Based on the definitions, it is straightforward to check that
\begin{equation}\label{re1}
  \log\|M_n^u(z,E,\omega)\|=-\frac{1}{2}\sum_{j=0}^{n-1}d(z+j\omega,\omega)+\log\|M_n^a(z,E,\omega)\|,
\end{equation}
where
\begin{equation}\label{defd}
  d(z,\omega)=\log|\lambda_a^2a(z+\omega)\tilde a(z)|.
\end{equation} It is also easily seen that $L_n^u(E,\omega)=L_n(\omega,E)\ge 0$, $L^u(E,\omega)=L(E,\omega)\ge 0$ and
\begin{equation}\label{re2}
  L(E,\omega)=L^a(E,\omega)-D,
\end{equation}
where
\[D=\int_{\mathbb{T}}\log|\lambda_aa(x)|dx=\int_{\mathbb{T}}\log|\lambda_a\tilde a(x)|dx=\frac{1}{2}\int_{\mathbb{T}}d(x,\omega)dx.\]
It is well known that $L(E,\omega)$ is a $C^{\infty}$ function of  $E$ on the resolvent set. So we only need to consider $E\in \mathscr{E}$, where
\[
\mathscr{E}:=[-2\lambda_a\|a(x)\|_{L^{\infty}(\mathbb{T})}-\lambda_v\|v(x)\|_{L^{\infty}(\mathbb{T})},\ 2\lambda_a\|a(x)\|_{L^{\infty}(\mathbb{T})}+\lambda_v\|v(x)\|_{L^{\infty}(\mathbb{T})}].
\]
Simple computations yield that for any irrational $\omega$ and $1\leq n\in\mathbb{N}$,
\[
\sup\limits_{E\in\mathscr{E},x\in\mathbb{T}}u^a_n(x,E,\omega)\leq M_0,
\]where
\[
M_0:=\log\left(3\lambda_a\|a\|_{L^{\infty}(\mathbb{T})}+2\lambda_v \|v\|_{L^{\infty}(\mathbb{T})}\right).
\]

If $a(x)\equiv 1$, then the Hamiltonian operators (\ref{jacobiequ}) become the following famous quasi-periodic analytic Sch\"odinger operators $S_{x,\omega}$ on $l^2(\mathbb{Z})$:
\begin{equation}\label{schequ}
(S_{x,\omega}\phi)(n)=\phi(n+1)+\phi(n-1)+\lambda_s v(x+n\omega)\phi(n)
,\ n\in\mathbb{Z}.\end{equation}
Then  the n-step transfer  matrix $M^s_n(x,E,\omega)=\prod_{k=n}^1M^s(x+k\omega,E)$, where
\begin{equation}\label{ms}
  M^s(x,E,\omega)=\left ( \begin{array}{cc}
 \lambda_s v(x)-E & -1 \\
 1& 0 \\
  \end{array}\right ),
\end{equation} has the complex analytic extension $ M^s_n(z,E,\omega)$ on  $\mathbb{T}_{\rho}$, which is analytic and unimodular such that Theorem \ref{erg} and the Avalanche Principle can be applied directly. In 2001, Goldstein and Schlag(\cite{GS}) obtained the conclusion which we state in Remark \ref{optimal} to prove  the Lyapunov exponent of the Schr\"odinger operators
 \begin{equation}\nonumber
  L^s(E,\omega):=\lim\limits_{n\to\infty}L_n^s(E,\omega)=\lim\limits_{n\to\infty}\frac{1}{n}\int\limits_{\mathbb{T}}\log\|M^s_n(x,E,\omega)\|dx
\end{equation}
 is H\"older continuous in $E$ for the strong Diophantine $\omega$. In that reference, they shew that  the keys to prove the continuity of the Lyapunov exponent are  two lemmas: the large deviation theorem (LDT for short) and the Avalanche Principle. We say  an  estimation is a LDT for \begin{equation}\label{uns}
  u^s_n(z,E,\omega)=\frac{1}{n}\log\|M^s_n(z,E,\omega)\|
\end{equation}  if it satisfies
\begin{equation}\label{abldt}\mes\{x:|u^s_n(x,E,\omega)-L^s_n(E,\omega)|> \delta\}<f\left(\delta,n\right).\end{equation}
They proved that $$f\left(\delta,n\right)=\exp\left(-c\delta^2n\right)$$ for strong Diophantine $\omega$, where $c$ depends only on the potential $v(x)$.  In 2002 Bourgain and Jitomirskaya shew in \cite{BJ} that for any irrational $\omega$,
\begin{equation}
  \label{ldtir}f\left(\delta,n\right)=\exp(-c\delta q),
\end{equation} where $0<\delta<1$ and $q$ is the a denominator of $\omega$'s continued fraction approximants $\{\frac{p_s}{q_s}\}_{s=1}^{\infty}$ satisfying $q\lesssim n$.  With the help of (\ref{ldtir}), they proved the joint continuity of $L^s(E,\omega)$ in $(E,\omega)$ at every $(E,\omega_0)$ if $\omega_0$ is irrational. On the other hand, people always think that the H\"older continuity of $L^s(E,\omega)$ in $E$ can not hold when $\bar\beta(\omega)\gtrsim \log\lambda_s$(Recently, Avila et al declared  the proof in the preparation reference  \cite{ALSZ}). So the question whether the H\"older continuity is true or not in the large coupling regime for the finite Liouville frequency $\omega$ is  eagerly anticipated in our field. The first answer  was \cite{YZ} in 2014. They proved that there exist constants $\lambda_0,\ N_0$ which only depend on $v$ and small absolute constants $c'$ and $c''$ such that if $\bar\beta(\omega)<c'$ and $\lambda_s>\lambda_0$, then for  any $n>N_0$
$$f\left(\frac{1}{100}\log\lambda_s,n\right)=\exp\left(-c''n\right)$$
and $L^s(E,\omega)$ is H\"older continuous in $E$. Very recently, that question was finished by Han and Zhang in \cite{HZ}. They proved that if the coupling number $\lambda_s$ is large and  $\bar\beta(\omega)\lesssim \log\lambda_s$, then the sharp LDT, which means
\begin{equation}\label{holderldt}
  f(\delta,n)=\exp(-c\delta n),
\end{equation} and  the H\"older continuity hold.

Compared to the Schr\"odinger cocycle, one of  the  distinguishing features of the Jacobi cocycle is that it is not $SL(2,\mathbb{C})$. It causes that Jitomirskaya and Marx  proved the week H\"older continuity of the Lyapunov exponent of  the analytic $GL(2,\mathbb{C})$ cocycles with Diophantine frequency in \cite{JM}. But in \cite{T}, I shew that the continuity of the Lyapunov exponent of the Jacobi cocycles defined in (\ref{leomega}) can be H\"older in $E$ with the strong Diophantine $\omega$. So far, from a technical perspective,  the strong Birkhoff Ergodic Theorem  is a necessary condition for applying the Avalanche Principle to the analytic quasi-periodic $GL(2,\mathbb{C})$ cocycles. The reason is that the Avalanche Principle can  be applied only to the matrices whose determinants are not larger than $1$.

In this paper, in order to extend the conclusion for the strong Diophantine frequency in  \cite{T} to the one for the finite Liouville frequency, we need to refer to a surprising discovery by Han and Zhang in \cite{HZ} that the  quantity $\mu(\Omega_1)$ in (\ref{eq:rieszrep1}) for the subharmonic function $u_n^s(\cdot,E,\omega)$ defined in (\ref{uns}) does not depend on the large $\lambda_s$. With its help, we apply Theorem \ref{erg} to obtain the sharp LDT (\ref{holderldt}) and the H\"older continuity of $L(E,\omega)$ in $E$ for any finite Liouville $\omega$ when $\lambda_v$ and $|\log\lambda_a|$ are in the large regimes. We also improve the applications of the Avalanche Principle to get the lengths  of the interval $I_E$ where $L(E,\omega)$ satisfies the H\"older condition for $E$. Moreover, if we consider the quasi-periodic analytic Schr\"odinger operators (\ref{schequ}), then we can also obtain the H\"older continuity of the Lyapunov exponent in $\omega$.~\\

Now we begin to state the details of our conclusions.  Choose
\be
  \label{Omega}\Omega=\{z:|\Re z|<1,\ |\Im z|<\rho\}\ee and \be
  \label{Omega1}\Omega_1=\{z:|\Re z|<\frac{2}{3},\ |\Im z|<\frac{\rho}{2}\}
\ee in Theorem \ref{erg}. Then, the LDT for the Jacobi cocycles is as follows:
\begin{theorem}\label{ldt}
There exist $\lambda_0=\lambda_0(v,\lambda_a,a)$ and $c_{v,a}=c_{v,a}(v,a)$ such that for any $\delta>0$, if $\beta(\omega)<c_{v,a}\min(\delta,|D|)$ and $\lambda_v>\lambda_0$, then
\begin{equation}\label{ldtue}\mes\{x:|u_n(x,E,\omega)-L_n(E,\omega)|> \delta\}<\exp(-c_1\delta^2n)+\exp(-c_2\delta n),\ \ \forall n\ge n_1,\end{equation}
where $n_1=n_1(\lambda_a,a,\lambda_v,v)$, $c_1=c_1(\lambda_a,a,\lambda_v,v)$ and $c_2=c_2(\lambda_a,a,\lambda_v,v)$.
\end{theorem}
\begin{remark}\label{re1-3}
  We can give the expressions of $\lambda_0$ and $c_{v,a}$ as follows:
  \[\lambda_0:=\max\left(\frac{\lambda_a\|a\|_{L^{\infty}(\Omega)}}{\|v\|_{L^{\infty}(\Omega)}},2\lambda_a\|a\|_{L^{\infty}(\Omega)}\epsilon_0^{-1}\right)\ \mbox{and}\  c_{v,a}:=\frac{1}{2CC(\Omega,\Omega_1)C_{v,a}},\]
  where $\epsilon_0$  defined in Lemma \ref{BG} depends only on $v$, $C$ is the  constant from Theorem \ref{erg}, $C(\Omega,\Omega_1)$ is a constant which depends only on $\Omega$ and $\Omega_1$, and $C_{v,a}:=\max\left(\log\frac{10\|v\|_{L^{\infty}(\Omega)}}{\epsilon_0},\log\frac{\|a\|_{L^{\infty}(\Omega)}}{\|a\|_{L^{\infty}(\Omega_1)}}\right)$.
\end{remark}
\begin{remark}\label{schldt}
 Due to the fact  that the Schr\"odinger operator is a special case of the Jacobi one, from now on to the end of this section, after every theorem for the Jacobi operators we will declare the corresponding conclusions for the Schr\"odinger ones in the next remark. The reason we do this is that, like the following LDT, we can get better results for  the Schr\"odinger ones: There exist $c_s=c_s(v):=\frac{1}{2CC_v}$ and $\lambda_0^s=\lambda_0^s(v):=2\epsilon_0^{-1}$ such that for any $\delta>0$, if $\beta(\omega)<c_s\delta$ and $\lambda_s>\lambda_0^s$, then for any positive $n$,
  \[\mes \{x:|u^s_n(x,E,\omega)-L^s_n(E,\omega)|> \delta\}<\exp(-\bar c_s\delta^2n),\]
    where   $\bar c_s=\bar c_s(v,\lambda):=\frac{c}{8M_0^sC_v}$, $C_v:=C(\Omega,\Omega_1)\log\frac{10\|v\|_{L^{\infty}(\Omega)}}{\epsilon_0}$, $C$ and $c$ are the absolute constants from Theorem \ref{erg}, $C(\Omega,\Omega_1)$ and $\epsilon_0$ are from Remark \ref{re1-3} and Lemma \ref{BG} respectively,   and $M_0^s:=\log\left(3+2\lambda_s \|v\|_{L^{\infty}(\mathbb{T})}\right)$.
\end{remark}~\\

If the Lyapunov exponent $L(E,\omega)$ is positive at one point $(E_0,\omega_0)$, then it is also positive on its neighborhood  where we can have  a better LDT, called the sharp one by Bourgain in \cite{B}:
\begin{theorem}\label{sldt}
 Assume $L(E_0,\omega_0)>0$. If $\beta(\omega_0)<c_{v,a}\min\left(\frac{L(E_0,\omega_0)}{15},|D|\right)$ and $\lambda_v>\lambda_0$, where the constants $c_{v,a}$ and $\lambda_0$ are defined in Theorem \ref{ldt}, then there exist $r_E=r_E(\lambda_a,a,\lambda_v,v,E_0,$
 $\omega_0)$  and $r_{\omega}=r_{\omega}(\lambda_a,a,\lambda_v,v,E_0,\omega_0)$ such that for any $|E-E_0|\leq r_E$ and $|\omega-\omega_0|\leq r_{\omega}$,
 \[\frac{3}{4}L(E_0,\omega_0)<L(E,\omega)<\frac{5}{4}L(E_0,\omega_0).\]
 Furthermore,  if $\beta(\omega)<\frac{1}{100}c_{v,a}L(E_0,\omega_0)$, then there exist $\bar c_{v,a}:=\frac{c}{C(\Omega,\Omega_1)C_{v,a}}$ and $\check n=(\lambda_a,a,\lambda_v,v,E_0,\omega_0)$ such that
 \begin{equation}\label{sldte}\mes\{x:|u_n(x,E,\omega)-L_n(E,\omega)|> \frac{1}{20} L(E,\omega)\}<\exp\left(- \frac{1}{12000}\bar c_{v,a}L(E,\omega)n\right),\ \ \forall n\ge \check n,\end{equation}where $c$ is the absolute constant from Theorem \ref{erg}, and  $C(\Omega,\Omega_1)$ and $C_{v,a}$ are from Remark \ref{re1-3}.
\end{theorem}
\begin{remark}\label{exp-r}
We apply the LDT (\ref{ldtue}) and the Avalanche Principle, not the continuity of Lyapunov exponent for the complex analytic cocycles with irrational $\omega$ proved in \cite{AJS}, to prove that $L(E,\omega_0)$ is positive  on the interval $[E_0-r_E,E_0+r_E]$. The benefit is that  we can calculate the expression of $r_E$:
\begin{equation}\label{exp-rE}
r_E=\frac{L(E_0,\omega_0)}{200\check n}\exp\big((1-\check n)M_0-2\check n|D|\big),
\end{equation}where $\check n$, which appears in (\ref{sldte}) and (\ref{exp-rE}), is defined in (\ref{checkn}). Due to the definition, it is easily seen that $r_E$ is a continuous function in $E_0$. Thus, if $L(E,\omega_0)$ is positive on $\mathcal{E}\times \{\omega_0\}$, then $r^i_E:=\inf_{E\in\mathcal{E}}r_E$ exists and is positive.
\end{remark}
\begin{remark}
  We can not get the expression of $r_{\omega}$, as it comes from the compactness in $E$ and the joint continuity of $L(E,\omega)$, which was proved in \cite{AJS}.
\end{remark}
\begin{remark}
  The parameters $3/4$, $5/4$ and $1/20$ can be replaced in turn by $1-\kappa_1$,$1+\kappa_2$ and $\kappa_3$, where $0<\kappa_1,\kappa_2,\kappa_3<1$. Then, the new constants $c_3^{\kappa}$ only differs from $c_3$ by a constant multiple of $400\kappa^{-2}$, $r_E(\kappa_1,\kappa_2)$ and $r_{\omega}(\kappa_1,\kappa_2)$ depend on $\kappa_1$ and $\kappa_2$, and $\check n_{\kappa_1,\kappa_2,\kappa_3}$  depends on $\kappa_1,\kappa_2$ and $\kappa_3$.
\end{remark}
\begin{remark}\label{schsldt}
For the Schr\"odinger  operators, we can  calculate the expressions of $r^s_E$ and $r^s_{\omega}$: Assume $L^s(E_0,\omega_0)>0$. If $\beta(\omega_0)<c_{s}\frac{L^s(E_0,\omega_0)}{15}$ and $\lambda_s>\lambda_0^s$, then there exist $\check n_s$ defined in (\ref{checkns}), $r^s_E=r^s_E(\lambda,v,E_0,\omega_0):=\frac{L^s(E_0,\omega_0)}{200\check n_{s}}\exp(-5M_0^s\check n_{s})$ and $r^s_{\omega}=r^s_{\omega}(\lambda,v,E_0,\omega_0):=\frac{L^s(E_0,\omega_0)}{400\max_{\mathbb{T}}(v'(x))\check n_{s}^2}\exp(-5M_0^s\check n_{s})$, where $M_0^s$ comes from Remark \ref{schldt}, such that  for any $|E-E_0|\leq r^s_E$ and  $|\omega-\omega_0|\leq r^s_{\omega}$,
 \[\frac{4}{5}L^s(E_0,\omega_0)<L^s(E,\omega)<\frac{6}{5}L^s(E_0,\omega_0).\]
 Furthermore,  if $\beta(\omega)<\frac{1}{100}c_{s}L(E_0,\omega_0)$, then there exists $c^s_3:=\frac{\bar c_s}{4\times 10^{3}}$ such that
 \begin{equation}\label{ssldte}\mes\{x:|u^s_n(x,E,\omega)-L^s_n(E,\omega)|> \frac{1}{20} L^s(E,\omega)\}<\exp(- c^s_3L^s(E_0,\omega_0)),\ \ \forall n\ge \check n_s.\end{equation}
 Here $c_s$ and $\bar c_s$ are both from Remark \ref{schldt}.
\end{remark}~\\

Due to the positive Lyapunov exponent and the sharp LDT (\ref{sldte}), the Avalanche Principle can be applied again to obtain the following H\"older continuity of Lyapunov exponent:
\begin{theorem}\label{holder}
Assume $L(E_0,\omega_0)>0$, $\beta(\omega_0)<c_{v,a}\min\left(\frac{L(E_0,\omega_0)}{15},|D|\right)$ and $\lambda_v>\lambda_0$, where the constants $c_{v,a}$ and $\lambda_0$ are defined in Theorem \ref{ldt}. There exists $\tau=\tau(v,a):=\frac{\bar c_{v,a}}{2\bar c_{v,a}+8\times 10^5}$, where $\bar c_{v,a}$ is from Theorem \ref{sldt}, such that  for any $E_1,E_2\in [E_0-r_E,E_0+r_E]$ and irrational  $\omega\in [\omega_0- r_{\omega},\omega_0+r_{\omega}]$ satisfying $\beta(\omega)< \frac{c_{v,a}L(E_0,\omega_0)}{100}$, it has
 \[\left|L(E_1,\omega)-L(E_2,\omega)\right|=\left|L^a(E_1,\omega)-L^a(E_2,\omega)\right|\leq \left(|E_1-E_2|\right)^{\tau}.\]
\end{theorem}
\begin{remark}\label{schholder}
For the Schr\"odinger cocycles, we can prove the  H\"older continuity  in $\omega$ as follows:
  Assume $L^s(E_0,\omega_0)>0$, $\beta(\omega_0)<\frac{c_sL^s(E_0,\omega_0)}{15}$ and $\lambda_s>\lambda_0^s$. There exists $\tau_s=\tau_s(v):=\frac{\bar c_s}{2\bar c_s+8\times 10^5}$ such that  for any $E_1,E_2\in [E_0-r_E^s,E_0+r_E^s]$ and $\omega_1,\omega_2\in [\omega_0-r^s_{\omega},\omega_0+r^s_{\omega}]$ satisfying $\max\left(\beta(\omega_1),\beta(\omega_2)\right)< \frac{c_sL^s(E_0,\omega_0)}{100}$,  it has
 \[\left|L^s(E_1,\omega_1)-L^s(E_2,\omega_2)\right|\leq |E_1-E_2|^{\tau_s}+|\omega_1-\omega_2|^{\tau_s}.\]
 Here the constants $c_s$, $\bar c_s$ and $\lambda_0^s$ come from Remark \ref{schldt}.
\end{remark}~\\

It is well-known that Screts-Spencer \cite{SS} gave a lower bound of the Lyapunov exponents of the Schr\"odinger operators in the large coupling regime. We  prove that the  similar result for the Jacobi ones also holds:
\begin{theorem}\label{posile}
For any $0<\gamma<1$, $E\in \mathscr{E}$ and irrational $\omega$, there exists $\lambda_p=\lambda_p(\lambda_a,a,v,\gamma)$ such that if $\lambda_v>\lambda_p$, then
\[L(E,\omega)>(1-\gamma)\log \lambda_v.\]
\end{theorem}
\begin{remark}
  The setting of $\lambda_p$ can be found in (\ref{lambda0}), which is a nondecreasing function of $\lambda_a$. Moreover, if $\lambda_v>\lambda_p$, then
  \[\left(1-\frac{\gamma}{2}\right)\log\lambda_v<L^a(E,\omega)<\left(1+\frac{\gamma}{2}\right)\log\lambda_v\ \mbox{and}\ D<\frac{\gamma}{2}\log\lambda_v.\]
  Indeed, the proof of this theorem is  obtained directly from the above two properties and (\ref{re2}).
\end{remark}
\begin{remark}
 Of course, we can also give the upper bound of $L(E,\omega)$, although it is useless in this paper. By (\ref{re2}),  we need to show that $D>-\frac{\gamma}{2}\log\lambda_v$. Define $D_a:=\exp\left(-\int_{\mathbb{T}}\log|a(x)|dx\right)$. Then, if $\lambda_v>\lambda_a^{-\frac{2}{\gamma}}D_a^{-1}$, which is a decreasing function of $\lambda_a$, then
 \[L(E,\omega)<(1+\gamma)\log\lambda_v.\]
\end{remark}

Thus, when $\lambda_v>\lambda_p$, due to Theorem \ref{holder}, we have that for any irrational $\omega$, if $\beta(\omega)<c_{v,a}\min\left(\frac{1-\gamma}{15}\log\lambda_v,|D|\right)$, then the H\"older continuity holds. Note  that  $D<\frac{\gamma}{2}\log\lambda_v$ and $L(E,\omega)>(1-\gamma)\log\lambda_v$ at this time. Thus, if $D$ is positive, then $|D|<\frac{1}{15}L(E,\omega)$ with $\gamma=\frac{1}{10}$. Conversely, if $D$ is negative, then we need $\beta(\omega)<-c_{v,a}D$ and $\lambda_v>\max(\lambda_p,\lambda_{\omega})$, where $\lambda_{\omega}:=\exp\left(\frac{15\beta(\omega)}{c_{v,a}(1-\gamma)}\right)$. Obviously, due to the definitions, $\lambda_p$ is a decreasing function of $\gamma$ with $\lim_{\gamma\to 0}\lambda_p\to\infty$, and $\lambda_{\omega}$ is an increasing function of $\gamma$ with $\lim_{\gamma\to1}\lambda_{\omega}\to\infty$. Therefore, there exists $0<\gamma_0<1$ such that $\lambda_p(\gamma_0)=\lambda_{\omega}(\gamma_0)$ and $\max(\lambda_p(\gamma_0),\lambda_{\omega}(\gamma_0))\leq \max(\lambda_p(\gamma),\lambda_{\omega}(\gamma))$ for any $0<\gamma<1$.

In summary, we prove that for any finite Liouville $\omega$, the H\"older continuity of Lyapunov exponent in $E$ holds with suitable $\lambda_a$ and $\lambda_v$ as follows:
\begin{theorem}\label{larholder}
 For  any finite Liouville $\omega$, if $\lambda_a>D_ae^{\frac{\beta(\omega)}{c_{v,a}}}$ and $\lambda_v>\lambda_p(\lambda_a,a,v,\frac{1}{10})$, or $\lambda_a<D_ae^{-\frac{\beta(\omega)}{c_{v,a}}}$ and $\lambda_v>\exp\left(\frac{15\beta(\omega)}{c_{v,a}(1-\gamma_0)}\right)$, then for any $|E_1-E_2|<r^i_E$,
\[\left|L(E_1,\omega)-L(E_2,\omega)\right|=\left|L^a(E_1,\omega)-L^a(E_2,\omega)\right|\leq \left(|E_1-E_2|\right)^{\tau},\]
where $r_E^i$ comes from Remark \ref{exp-r}, $c_{v,a}$ is from Remark \ref{re1-3} and $\tau$ is defined in Theorem \ref{holder}.
\end{theorem}
\begin{remark}
  \label{schplholder}For the Schr\"odinger operators in the large coupling regime, $r_E^s$ and $r^s_{\omega}$ only depend on $\lambda_s$: For any finite Liouville $\omega$, if $\lambda_s>\max\left[\left(\frac{20\|v\|_{L^{\infty}(\Omega)}}{\epsilon_0^2}\right)^{50},\ \exp\left(\frac{16}{c_s}\beta(\omega)\right), 5\max_{x\in\mathbb{T}}|v'(x)|\right]$, where $\epsilon_0$ and $c_s$ are from Lemma \ref{BG} and Remark \ref{schldt} respectively, then for any $E,E'$ and irrational $\omega'$ satisfying $|E-E'|<\lambda_s^{-800}$, $|\omega-\omega'|<\lambda_s^{-800}$ and $\beta(\omega')\leq \beta(\omega)$, it has
  \begin{equation}
    \label{schlarholder}
    |L^s(E,\omega)-L^s(E',\omega')|<|E-E'|^{\tau_s}+|\omega-\omega'|^{\tau_s},
  \end{equation}where $\tau_s$ is defined in Remark \ref{schholder}.
\end{remark}~\\

This paper is organized as follows.  In Section 2,   the strong Birkhoof Ergodic Theorem for subharmonic functions with irrational shift on the Torus is proved. Then we study the positive Lyapunov exponent of the Jacobi operators in the large potential coupling regimes, and obtain the independences of  the quantities $\mu(\Omega_1)$ in (\ref{eq:rieszrep1}) for two subharmonic functions $u_n^a(\cdot,E,\omega)$ and $d(\cdot,\omega)$, which are defined in (\ref{uan}) and (\ref{defd}) respectively,  on the coupling numbers by Han and Zhang's method to make Theorem \ref{erg} applied to the Jacobi operators with suitable frequencies and coupling numbers in Section 3. It help us in getting Theorem \ref{ldt}, the LDT for the Jacobi cocycles, in Section 4. Combining this LDT with the Avalanche Principle, we prove that the positive Lyapunov exponent can be extended from one point to an interval, and calculate its length in Section 5. Finally, the proofs of the rest theorems, Theorem \ref{sldt}, \ref{holder} and \ref{larholder}, are presented in the last section. In additional,  the results of the Schr\"odinger operators stated in the remarks are proved under the proofs of the corresponding theorems of the Jacobi ones.

 \section{The Ergodic Theorem for Subharmonic Functions with Shift}
 Let $\{x\}=x-[x]$. For any positive integer $q$, complex number $\zeta=\xi+i\eta$ and $0\leq x<1$, define
 \begin{equation}
   f_{q,\zeta}(x)=\sum_{0\leq k<q}\log |\{x+\frac{k}{q}\}-\zeta|,\  F_{q,\zeta}(x)=\sum_{0\leq k<q}\log |\{x+k\omega\}-\zeta|\ \mbox{and}\  I(\zeta)=\int_0^1\log |y-\zeta |dy.
 \end{equation}
 Also define
 $$\|x\|=\min_{n\in \mathbb{Z}}|x+n|,$$
 and the distance on $\mathbb{T}$:
 \[dist(x,y)=\|x-y\|.\]
Given $x\in [0,1)$ and $q=1,2,\cdots$. Set $\mathcal{S}=\{\{x+\frac{k}{q}\}: 0\leq k <q\}$. We enumerate $\mathcal{S}$ as
\begin{enumerate}
  \item[{\rm{(1)}}]$0\leq \theta_0<\theta_1<\cdots < \theta_{q-1}\leq 1,\ \theta_j=\{x+\frac{k_j}{q}\}$;
  \item[{\rm{(2)}}]$\theta_{j+1}=\theta_j+\frac{1}{q}$.
\end{enumerate}
For any $\xi\in[0,1)$, find the integers $j^+(x,q,\xi)$ and $j^-(x,q,\xi)$ such that if $\theta_0\leq \xi <\theta_{q-1}$, then
\[\theta_{j^-(x,q,\xi)}\leq \xi <\theta_{j^+(x,q,\xi)}\ \mbox{and}\ j^+(x,q,\xi)-j^-(x,q,\xi)=1,\]
else $j^-(x,q,\xi)=q-1$ and $j^+(x,q,\xi)=0$.
Let $k^-(x,q,\xi)=k_{j^-(x,q,\xi)}$ and $k^+(x,q,\xi)=k_{j^+(x,q,\xi)}$. Then, we have
 \begin{lemma}\label{lemra}
 There exists an absolute constant $C$ such that
 \[\left |\sum_{0\leq k<q,k\not = k^{\pm}(x,q,\xi)}\log |\{x+\frac{k}{q}\}-\zeta|-qI(\zeta)\right |\leq C\log q,
 \] where $C$ is an absolute constants.
 \end{lemma}
 \begin{proof}
  Using the above notations one has
  \[\frac{1}{q}\sum_{0\leq k<q,k\not = k^{\pm}(x,q,\xi)}\log |\{x+\frac{k}{q}\}-\zeta|=\frac{1}{q}\sum_{0\leq j <q,j\not = j^{\pm}(x,q,\xi)}\log |\theta_j-\xi-i\eta|.\]
 Also we have for any $0\leq j<j^-(x,q,\xi)$,
  \begin{equation}\nonumber
   \left |(\theta_{j^-(x,q,\xi)}-\xi)-(\theta_j-\xi)\right |\leq \frac{j^-(x,q,\xi)-j}{q},
  \end{equation}
  and for any $j^+(x,q,\xi)<j\leq q-1$,
   \begin{equation}\nonumber
    \left |(\theta_{j^+(x,q,\xi)}-\xi)-(\theta_j-\xi)\right |\leq \frac{j-j^+(x,q,\xi)}{q}.
   \end{equation}
   Note that if $1\leq j<j^-(x,q,\xi)$, then
   \[ \frac{1}{q}\log |\theta_j(x,q)-\xi-i\eta|<\int_{\theta_j(x,q)-1}^{\theta_j(x,q)}\log |y-\xi-i\eta|dy< \frac{1}{q}\log |\theta_{j(x,q)-1}-\xi-i\eta|,                                                    \]
   and if $j^+(x,q,\xi)<j\leq q-2$, then
   \[ \frac{1}{q}\log |\theta_j(x,q)-\xi-i\eta|<\int_{\theta_j(x,q)}^{\theta_j(x,q)+1}\log |y-\xi-i\eta|dy< \frac{1}{q}\log |\theta_{j(x,q)+1}-\xi-i\eta|.                                                    \]
   Thus
  \begin{eqnarray} \left |\frac{1}{q}\sum_{0\leq j <q,j\not = j^{\pm}(x,q,\xi)}\log |\theta_j-\xi-i\eta|-I(\zeta)\right |
  &\leq &\left |\sum_{0\leq j <q,j\not = j^{\pm}(x,q,\xi)} \frac{1}{q}\log |\theta_j-\xi-i\eta|-\sum_{j=0}^{q-1}\int_{\theta_j}^{\theta_{j+1}}\log |y-\xi-i\eta|dy\right |
 \nonumber\\
  &\leq &  -4\int_{|y|<\frac{1}{q}}\log |y|dy<\frac{C}{q}\log q.\nonumber
  \end{eqnarray}
 \end{proof}
\begin{remark}\label{rem001}
  Let $|\{x+k'_0/q\}-\xi|=\min_{k=1}^{q_s} |\{x+k/q\}-\xi|$, then $k'_0=k^-(x,q,\xi)$ or $k'_0=k^+(x,q,\xi)$. Note that $|\{x+k^-(x,q,\xi)/q\}-\xi|+|\{x+k^+(x,q,\xi)/q\}-\xi|=\frac{1}{q}$. Then the bigger one of these two numbers is larger than $\frac{1}{2q}$. Thus
  \[\left |\sum_{0\leq k<q,k\not = k'_0}\log |\{x+\frac{k}{q}\}-\zeta|-qI(\zeta)\right |\leq C\log q,
 \] where $C$ is an absolute constant.
\end{remark}

Let $\omega$ be irrational and $\left\{\frac{p_s}{q_s}\right\}_{s=1}^{\infty}$ be its continued fraction approximants. Then by (\ref{irtor0}), it has
\begin{equation}\label{irtor}
 \frac{k}{q_s(q_{s+1}+q_s)}<|k\omega-\frac{k p_s}{q_s}|<\frac{k}{q_sq_{s+1}}\leq \frac{1}{q_{s+1}},\ \ 0<k< q_s.
\end{equation}
\begin{lemma}\label{lemir}
 Let $|\{x+k_0\omega\}-\xi|=\min_{k=1}^{q_s} |\{x+k\omega\}-\xi|$,  then
 \[\left |F_{q_s,\zeta}(x)-q_sI(\zeta)\right |\leq C\log q_s+\left |\log |\{x+k_0\omega\}-\zeta|\right |.
 \]
 \end{lemma}
 \begin{proof}
We declare that  if there exists $0\leq j<q_s$ such that $|\{x+j\omega\}-\xi|\leq \frac{1}{2q_s}-\frac{1}{q_{s+1}}$, then $j=k_0$.  Actually, if $|\{x+j\omega\}-\xi|\leq \frac{1}{2q_s}-\frac{1}{q_{s+1}}$ and $j \not =k_0$, then $|\{x+k_0\omega\}-\xi|\leq |\{x+j\omega\}-\xi|\leq \frac{1}{2q_s}-\frac{1}{q_{s+1}}$, which implies
\[\left|\{x+k_0\omega\}-\{x+j\omega\}\right|\leq \frac{1}{2q_s}-\frac{2}{q_{s+1}}.\]
 By (\ref{irtor}), we have
\[\left |\{x+j\frac{p_s}{q_s}\}-\{x+k_0\frac{p_s}{q_s}\}\right |<\frac{1}{q_s}.\] It is a contraction. Thus there is at  most  one integer $0\leq k_0<q_s$ such that $|\{x+k_0\omega\}-\xi|< \frac{1}{2q_s}-\frac{1}{q_{s+1}}$ and
\begin{equation}\label{10004}
 |\{x+k\omega\}-\xi|>\frac{1}{2q_s}-\frac{1}{q_{s+1}}>\frac{1}{4q_{s}},\ \ k\not =k_0.
\end{equation}

Due to (\ref{irtor}),
  we have for any $0\leq k< q_s$,
  \begin{equation}\label{10003} \left | |\{x+k\omega\}-\xi|-|\{x+k\frac{p_s}{q_s} \}-\xi |\right |<\frac{1}{q_{s+1}},  \end{equation}
  or there exists only one integer $0\leq m<q_s$ such that
  \[\left | |\{x+m\omega\}-\xi|+|\{x+m\frac{p_s}{q_s} \}-\xi |\right |>1-\frac{1}{q_{s+1}},\]
  and the others satisfies (\ref{10003}). Notice that $\{\{x+\frac{k}{q_s}\}:0\leq k<q_s\}=\{\{x+\frac{kp_s}{q_s}\}:0\leq k<q_s\}$. So, by Remark \ref{rem001} , we have
\[ |\{x+k\frac{p_s}{q_s}\}-\xi|>\frac{1}{2q_{s}},\ \ k\not =k'_0,\]
Combining it with (\ref{10003}), we  have
\begin{eqnarray}
    \left |\sum_{0\leq k<q_s,k\not = k_0,k'_0}\left [\log |\{x+k\omega\}-\zeta|-\log |\{x+kp_s/q_s\}-\zeta|\right ]\right |&\leq &\sum_{0\leq k<q_s,k\not = k_0,k'_0}\left |\log \left [1+\frac{|\{x+k\omega\}-\{x+kp_s/q_s\}|}{|\{x+kp_s/q_s\}-\zeta|}\right ]\right |\nonumber\\
    &\leq & C\sum_{l=1}^{q_s}\frac{q^{-1}_{s+1}}{lq_s^{-1}}\leq C\frac{q_s\log q_s}{q_{s+1}}\leq C\log q_s.\nonumber
  \end{eqnarray}
By Remark \ref{rem001}, it implies that
\begin{equation}\label{10005}
\left |\sum_{0\leq k<q,k\not =k_0, k'_0}\log |\{x+k\omega\}-\zeta|-q_sI(\zeta)\right |\leq C\log q.
\end{equation}
Then, combining it with (\ref{10004}), we complete the proof.
 \end{proof}

\begin{lemma}\label{lemlq}Let $n=lq_s<q_{s+1}$, then
  \begin{equation}
  |F_{n,\zeta}(x)-n I(\zeta)|<Cl\log q_s+| \log D(x-\xi,-\omega,lq_s)|+2 \beta n.
\end{equation}
\end{lemma}

\begin{proof}
Define $x_h=x+hq_s \omega$ and $|\{x_h+k_h\omega\}-\xi|=\min_{k=0}^{q_s-1}|\{x_h+k\omega\}-\xi|$.
Then due to Lemma \ref{lemir}, we have
\[
|F_{lq_s,\zeta}(x)-lq_s I(\zeta)|\leq \sum_{h=0}^{l-1}  |F_{q_s,\zeta}(x_h)-lq_s I(\zeta)|     \leq\sum_{h=0}^{l-1}\left |\log |\{x_h+k_h\omega\}-\zeta|\right |+Cl\log q_s.
 \]
Note that
\begin{equation}\label{10012}  \frac{1}{2q_{s+1}}<\frac{1}{q_s+q_{s+1}}<|q_s\omega-p_s|<\frac{1}{q_{s+1}}.\end{equation}
Define $Q=[\frac{q_{s+1}}{q_s}]$ and
 let $j$ be the  number such that $ |\{x_j+k_j\omega\}-\xi|<\frac{1}{4q_{s+1}}$. Then by (\ref{10012}) and the declaration in the proof of Lemma \ref{lemir}, we have for any $ j-2Q+1\leq h<j$ and $j<h\leq j+2Q-1$,
 \[ |\{x_h+k_h\omega\}-\xi|>\frac{1}{4q_{s+1}}.\]
 Thus there are at most one point which is small than $\frac{1}{4q_{s+1}}$. Recall that $n=lq_s$. Then we have
\begin{eqnarray}
  |F_{lq_s,\zeta}(x)-lq_s I(\zeta)|&\leq &\sum_{h=0}^{l-1}\left |\log |\{x_h+k_h\omega\}-\zeta|\right |+Cl\log q_s \nonumber\\
  &\leq &| \log D(x-\xi,-\omega,lq_s)|+Cl\log q_s +l\left |\log \frac{1}{4q_{s+1}}\right | \nonumber\\
  &\leq & | \log D(x-\xi,-\omega,lq_s)|+Cl\log q_s +2l\log q_{s+1} \nonumber\\
  & \leq & | \log D(x-\xi,-\omega,lq_s)|+Cl\log q_s+2\frac{n}{q_s}\beta q_s\nonumber\\
  & \leq & | \log D(x-\xi,-\omega,lq_s)|+Cl\log q_s+2\beta n.\nonumber
\end{eqnarray}
\end{proof}

\begin{lemma}[Lemma 3.2 in \cite{GS}]
  Let $\Omega\subset \mathbb{T}$ be an arbitrary finite set. Then
  \[\int_{\mathbb{T}}\exp\left (\sigma |\log dist (x,\Omega)|\right )dx\leq \frac{2^{\sigma}}{1-\sigma}(\sharp \Omega)^{\sigma}\]
  for any $0<\sigma<1$.
\end{lemma}

\begin{lemma}
  Let $n=lq_s$. Then for any $0<\sigma<1$ ,
\[
  \int_{\mathbb{T}}\exp\left (\sigma|F_{n,\zeta}(x)-n I(\zeta)|\right )dx <\exp\left (5\sigma \beta n\right ).
\]
\end{lemma}
\begin{proof}
  Set $\Omega=\{-m\omega:0\leq m<lq_s\}$. Then $\sharp \Omega=lq_s$ and $dist (x-\xi,\Omega)=D(x-\xi,-\omega,lq_s)$. Thus
  \[\int_{\mathbb{T}}\exp\left (\sigma |\log D(x-\xi,-\omega,lq_s)|\right )dx=\int_{\mathbb{T}}\exp\left (\sigma |\log dist (x,\Omega)|\right )dx\leq  \frac{2^{\sigma}}{1-\sigma} (lq_s)^{\sigma}.\]
  By Lemma \ref{lemlq}, we have
  \[ \int_{\mathbb{T}}\exp\left (\sigma|F_{n,\zeta}(x)-n I(\zeta)|\right )dx\leq \exp(2C\sigma \log (lq_s) + C\sigma l\log q_s+2\sigma \beta n)<\exp(5\sigma \beta n). \]
\end{proof}
\begin{remark}
  It is easily seen that there exists a constant $\hat C(\zeta)$ such that for any $n>0$ and $0<\sigma<1$,
  \[
  \int_{\mathbb{T}}\exp\left (\sigma|F_{n,\zeta}(x)-n I(\zeta)|\right )dx <\exp\left (\sigma \hat C(\zeta) n\right ).
\]So, the above lemmas show that for any $n=lq_s<q_{s+1}$,  the large constant $\hat C(\zeta)$ can be changed by $5\beta$. It is also obvious that if $n=lq_s$ is large, then $\hat C(\zeta)$ can be changed by $5\bar\beta$.
\end{remark}~\\

Now for any  $n$, there exist  $q_s$  and $q_{s+1}$ such that $q_s\leq n<q_{s+1}$. Let $n=l_sq_{s}+r_s$, where
$l_s=[\frac{n}{q_{s}}]$ and $ 0\leq r_s=n-l_sq_s<q_{s}$. Then,
\begin{eqnarray}
\int_0^1\exp(\sigma |F_{n,\zeta}(x)-n I(\zeta)|)dx &\leq & \left
[\int_0^1\exp(2 \sigma |F_{l_sq_{s},\zeta}(x)-l_sq_{s}
I(\zeta)|)dx\right ]^{\frac{1}{2}}\times \left [\int_0^1\exp(2 \sigma
|F_{r_s,\zeta}(x)-r_s
I(\zeta)|)dx\right ]^{\frac{1}{2}}\nonumber\\
&\leq & \exp(5\sigma\beta n) \left [\int_0^1\exp(2\sigma |F_{r_s,\zeta}(x)-r_s
I(\zeta)|)dx\right ]^{\frac{1}{2}}.\nonumber
\end{eqnarray}
Let $r_{s}=l_{s-1}q_{s-1}+r_{s-1}$, where
$l_{s-1}=[\frac{r_{s}}{q_{s-1}}]$ and $ 0\leq r_{s-1}=r_{s}-l_{s-1}q_{s-1}<q_{s-1}$. Then
\begin{eqnarray}
   \left
[\int_0^1\exp(2\sigma  |F_{r_s,\zeta}(x)-r_s I(\zeta)|)dx\right
]^{\frac{1}{2}}&\leq & \left
[\int_0^1\exp(2^2 \sigma |F_{l_{s-1}q_{s-1},\zeta}(x)-l_{s-1}q_{s-1}
I(\zeta)|)dx\right ]^{\frac{1}{2^2}} \nonumber\\
&&\ \ \times\left [\int_0^1\exp(2^2 \sigma
|F_{r_{s-1},\zeta}(x)-r_{s-1}
I(\zeta)|)dx\right ]^{\frac{1}{2^2}}\nonumber\\
&\leq & \exp\left (5\sigma \beta r_s \right)\left
[\int_0^1\exp(2^2\sigma |F_{r_{s-1},\zeta}(x)-r_{s-1} I(\zeta)|)dx\right
]^{\frac{1}{2^2}}\nonumber.
\end{eqnarray}
We use induction here. Let $r_{s-i+1}=l_{s-i}q_{s-i}+r_{s-i}$, where
$l_{s-i}=[\frac{r_{s-i+1}}{q_{s-i}}]$ and $ 0\leq r_{s-i}=r_{s-i+1}-l_{s-i}q_{s-i}<q_{s-i}$. Then
\begin{eqnarray}
   \left
[\int_0^1\exp(2^{i}\sigma |F_{r_{s-i+1},\zeta}(x)-r_{s-i+1} I(\zeta)|)dx\right
]^{\frac{1}{2^i}}&\leq & \left
[\int_0^1\exp(2^{i+1}\sigma |F_{l_{s-i}q_{s-i},\zeta}(x)-l_{s-i}q_{s-i}
I(\zeta)|)dx\right ]^{\frac{1}{2^{i+1}}}\nonumber\\
&&\ \ \times \left [\int_0^1\exp(2^{i+1}\sigma
|F_{r_{s-i},\zeta}(x)-r_{s-i}
I(\zeta)|)dx\right ]^{\frac{1}{2^{i+1}}}\nonumber\\
&\leq & \exp\left (5\sigma\beta r_{s-i+1} \right)\times \left
[\int_0^1\exp(2^{i+1}\sigma |F_{r_{s-i},\zeta}(x)-r_{s-i} I(\zeta)|)dx\right
]^{\frac{1}{2^{i+1}}}\nonumber.
\end{eqnarray}
Note that for any irrational $\omega$, there exists an absolute constant $\bar C>1$ such that $q_{s+1}>\bar C q_s$. Thus, there exists $m\ge 0$ such that $\bar C^{-m}\leq \beta$. Therefore, if
$\zeta\in \Omega'$, where $\Omega'$ is a compact subregion of $\mathbb{C}$, then
\begin{eqnarray}
\int_0^1\exp(\sigma |F_{n,\zeta}(x)-n I(\zeta)|)dx &\leq & \exp\left
[ 5\sigma \beta n+5\sigma \beta (r_s+r_{s-1}+\cdots+r_{s-m+1})\right ]\nonumber\\
&&\ \ \times \left [\int_0^1\exp(2^{m+1} \sigma
|F_{r_{s-m},\zeta}(x)-r_{s-m}
I(\zeta)|)dx\right ]^{\frac{1}{2^{m+1}}}\nonumber\\
&\leq & \exp\left
[ 5\sigma \beta n+5\sigma \beta (q_s+q_{s-1}+\cdots+q_{s-m+1})+\hat C(\zeta)\sigma r_{s-m}\right ]\nonumber\\
&\leq & \exp\left
[ 5\sigma \beta n+5\sigma \beta (q_s+\bar C^{-1}q_s+\cdots+\bar C^{-m+1}q_{s})+\hat C(\zeta)\sigma q_{s-m}\right ]\nonumber\\
&\leq & \exp\left
[ 5\sigma \beta n+\tilde{C}\sigma \beta q_s+\hat C(\zeta)\bar C^{-m}\sigma q_{s}\right ]\nonumber\\
&\leq & \exp\left
[ 5\sigma \beta n+\tilde{C}\sigma \beta q_s+\hat C(\zeta)\beta \sigma q_{s}\right ]<\exp(C\sigma \beta n),\nonumber
\end{eqnarray}
where $\tilde{C}=\sum_{k=0}^{\infty}\bar C^{-k}<\infty$  is an absolute constant and $C=C(\Omega')$.
Thus
\begin{lemma}
There exists $c=c(\omega,\bar C)$ such that for any positive $n$  and  $0<\sigma\leq c$, we have
\[\int_0^1\exp(\sigma |F_{n,\zeta}(x)-n I(\zeta)|)dx<\exp(C\sigma \beta n).\]
\end{lemma}
\begin{remark}\label{constant}
It is easy to see that the constant $c=2^{\log_{\bar C} \beta}$ is an increasing function of $\beta$. So, $c_0:=2^{\log_{\bar C} \beta_G}\le c(\omega,\bar C)$ for any irrational $\omega$, where $\beta_G=\beta(\omega_G)$ and $\omega_G=\frac{\sqrt 5-1}{2}$.
\end{remark}
\begin{remark}\label{remforhar}
  Note that $\log |x|$ is a subharmonic function. Thus, if $h$ is a 1-periodic harmonic function defined on a neighborhood of real axis, then for any positive $n$ and  $0<\sigma\leq c_0$, we have
\begin{equation}\label{12014}
 \int_0^1\exp(\sigma | \sum_{k=1}^n h(\{x+k\omega\})-n \int_0^1h(y)dy|)dx<\exp(C\sigma \beta n).
\end{equation}
\end{remark}

Now let us recall the following Riesz's theorem proved in \cite{GS1}:
\begin{lemma}
\label{lem:riesz} Let $u:\Omega\to \IR$ be a subharmonic function on
a domain $\Omega\subset\IC$. Suppose that $\partial \Omega$ consists
of finitely many piece-wise $C^1$ curves. There exists a positive
measure $\mu$ on~$\Omega$ such that for any $\Omega_1\Subset \Omega$
(i.e., $\Omega_1$ is a compactly contained subregion of~$\Omega$),
\begin{equation}
\label{eq:rieszrep} u(z) = \int_{\Omega_1}
\log|z-\zeta|\,d\mu(\zeta) + h(z),
\end{equation}
where $h$ is harmonic on~$\Omega_1$ and $\mu$ is unique with this
property. Moreover, $\mu$ and $h$ satisfy the bounds \be
\mu(\Omega_1) &\le& C(\Omega,\Omega_1)\,(\sup_{\Omega} u - \sup_{\Omega_1} u), \label{21002} \\
\|h-\sup_{\Omega_1}u\|_{L^\infty(\Omega_2)} &\le&
C(\Omega,\Omega_1,\Omega_2)\,(\sup_{\Omega} u - \sup_{\Omega_1} u)
\label{21003} \ee for any $\Omega_2\Subset\Omega_1$.
\end{lemma}
\noindent
Notice that the ergodic measure for the shift on the Torus is the Lebesgue measure and $m(\mathbb{T})=1$. Then, $<u>=\int_{\mathbb{T}}u(x)dx$, and
\[\sum_{k=1}^n u(x+k\omega)-n<u>=\sum_{k=1}^n\int_{\Omega_1} \log |\{x+k\omega\}-\zeta|d\mu(\zeta)-n\int_{\Omega_1} I(\zeta)d\mu(\zeta)+\sum_{k=1}^nh(\{x+k\omega\})-n\int_0^1h(y)dy.\]
Recall that
\[\sum_{k=1}^n\int_{\Omega_1} \log |\{x+k\omega\}-\zeta|d\mu(\zeta)=\int_{\Omega_1} F_{n,\zeta}(x)d\mu(\zeta).\]
Then
\begin{eqnarray}
  \int_0^1\exp\left (\sigma |\sum_{k=1}^n u(x+k\omega)-n<u>|\right )dx & \leq &\left [\int_0^1\exp\left (\sigma \left |\int_{\Omega_1} (F_{n,\zeta}(x)-n I(\zeta))d\mu(\zeta)\right |\right )dx\right ]^{\frac{1}{2}}\nonumber\\
   &&\ \ \times \left [\int_0^1\exp\left (\sigma \left |\sum_{k=1}^nh(\{x+k\omega\})-n\int_0^1h(y)dy\right |\right )dx\right ]^{\frac{1}{2}}\nonumber.
\end{eqnarray}
Since $\exp(\sigma \cdot)$ is a convex function, the Jensen's inequality implies that
\begin{eqnarray}
  \int_0^1\exp\left (\sigma \left |\int_{\Omega_1} (F_{n,\zeta}(x)-n I(\zeta))d\mu(\zeta)\right |\right )dx &\leq & \int_0^1\int_{\Omega_1}\exp\left (\sigma \mu(\Omega_1)\left | F_{n,\zeta}(x)-n I(\zeta)\right |\right )\frac{d\mu(\zeta)}{\mu(\Omega_1)}dx\nonumber\\
  &=&\int_{\Omega_1}\int_0^1\exp\left (\sigma \mu(\Omega_1)\left | F_{n,\zeta}(x)-n I(\zeta)\right |\right )dx\frac{d\mu(\zeta)}{\mu(\Omega_1)}\nonumber\\
  &\leq & \int \exp(C\sigma \mu(\Omega_1)\beta n)\frac{d\mu(\zeta)}{\mu(\Omega_1)}\leq \exp(C\sigma \mu(\Omega_1)\beta n).\nonumber
 \end{eqnarray}
Thus, combining it with (\ref{12014}), we have for any $0<\sigma\leq \frac{c_0}{\mu(\Omega_1)}$ and  $\omega$,
\[\int_0^1\exp\left (\sigma |\sum_{k=1}^n u(x+k\omega)-n<u>|\right )dx< \exp(C\sigma \mu(\Omega_1)\beta n).\]
Recall the Markov's inequality: For any  measurable extended real-valued function $f(x)$ and $\epsilon >0$,we have
\[\mes \left (\{x\in \mathbb{X}:|f(x)|\ge \epsilon\} \right ) \leq \frac{1}{\epsilon}\int_{\mathbb{X}} |f|dx.\]
Let $f(x)=\exp\left (\sigma |\sum_{k=1}^n u(x+k\omega)-n<u>|\right )$ and $\epsilon=\exp(\sigma \delta n)$, then
\begin{eqnarray}
 \mes\left (\left \{x\in \mathbb{X}:|\sum_{k=1}^n u(x+k\omega)-n<u>|>\delta n\right \}\right )&=& \mes \left (\left \{x\in \mathbb{X}:\exp\left (\sigma |\sum_{k=1}^n u(x+k\omega)-n<u>|\right ) \ge \exp(\sigma \delta n)\right \} \right )\nonumber\\
 &\leq & \exp(-\sigma \delta n+C\sigma \mu(\Omega_1)\beta n)
\end{eqnarray}
Therefore,  if $\beta<\frac{\delta}{2C\mu(\Omega_1)}$, then we complete the proof of Theorem \ref{erg}.
\begin{remark}
  \label{barbeta}
  In this proof, we use the constant $\beta$ only to make the inequality $\log q_{s+1}<\beta q_s$ hold for any $s>0$. Thus, if we change $\beta$ by $\bar\beta$, then for large $n$, this inequality and Theorem \ref{erg} still hold. But in this condition, the integer $m$ and the constant $c$ will depend on $\bar\beta$.
\end{remark}

\section{Positive Lyapunov Exponents and Strong Birkhoof Ergodic Theorem for Jacobi Operators}
The fact that the determinant of  $M^s_n(x,E,\omega)$, which is the transfer matrix of the Schr\"odinger operators and analytic in $x$, is always $1$ makes $u_n^s(x,E,\omega)=\frac{1}{n}\log\|M^s_n(x,E,\omega)\|$ have the upper and lower bounds. Correspondingly, $M_n(x,E,\omega)$ is the one of the Jacobi operators. But it is not analytic, and  its determinant and the  logarithm of its norm have no  bounds. Therefore, in the introduction  we define the matrix $M^a_n(x,E,\omega)$, which is analytic, so that Theorem \ref{erg} can  be applied to it in this section. We also define the matrix $M^u_n(x,E,\omega)$, whose determinant is $1$, so that the Avalanche Principle can  be applied to it in Section $5$ and $6$. On the other hand, what we want to present in the theorems is the fine properties associated with $M_n(x,E,\omega)$, such as  Theorem \ref{ldt} to \ref{larholder}. Therefore, in the rest of this paper, we need to transform these matrices into each another very often. We hope that these explanations can alleviate readers' confusion caused by these multiple matrices.

Theorem \ref{posile} and Lemma \ref{muomu} we prove in this section  have  similar versions for the Schr\"odinger operators (\ref{schequ}) in \cite{SS} and \cite{HZ} and there is no essential difference in their proofs. The reason why we do not omit this part is that  the details especially the settings of $\lambda_i$ are very important to our paper.~\\

Now we start to  study the Lyapunov exponents of the  analytic quasi-periodic Jacobi cocycles.
Choose $\Omega$ in Lemma \ref{lem:riesz} as (\ref{Omega}) and let $\lambda_v>\lambda_1:=\frac{\lambda_a\|a\|_{L^{\infty}(\Omega)}}{\|v\|_{L^{\infty}(\Omega)}}$. Then
\begin{equation}\label{upb}
  \sup\limits_{E\in\mathscr{E},z\in\Omega}u^a_n(z,E,\omega)\leq \log(5\lambda_v\|v\|_{L^{\infty}(\Omega)}).
\end{equation}
And if $\lambda_v>\lambda_2:=\left(5\|v\|_{L^{\infty}(\Omega)}\right)^{\frac{1}{\gamma}}$, then
\begin{equation}\label{suplea}
  L^a(E,\omega)\leq \sup\limits_{E\in\mathscr{E},z\in\Omega}u^a_n(z,E,\omega)\leq (1+\gamma )\log\lambda_v.
\end{equation}
By the way, it is easy to see that if $\lambda_v>\max(\lambda_2,\lambda_D)$, where $\lambda_D=\frac{\left(\lambda_a\|a\|_{L^{\infty}(\Omega)}\right)^{\frac{2}{\gamma}}}{\|v\|_{L^{\infty}(\Omega)}}$, then
\begin{equation}\label{supd}
D\leq \log\left(\lambda_a\|a\|_{L^{\infty}_{\mathbb{T}}}\right)\leq \log\left(\lambda_v\|v\|_{L^{\infty}_{\mathbb{T}}}\right)^{\frac{\gamma}{2}}\leq \frac{\gamma}{2}\log\lambda_v+\frac{\gamma^2}{2}\lambda_v<\gamma\log\lambda_v.
\end{equation}~\\

To estimate the lower bound of $\sup\limits_{z\in\Omega_1}u^a_n(z)$, we need the following lemma for the complex analytic function $v(z)$:
\begin{lemma}[Lemma 14.5 in \cite{BG}]\label{BG}
  For all $0<\delta <\rho$, there is  an $\epsilon_0=\epsilon_0(v)$ such that
  \[
    \inf_{E_1}\sup_{\frac{\delta}{2}<y<\delta}\inf_{x\in [0,1]}|v(x+iy)-E_1|>\epsilon_0.
  \]
\end{lemma}
\noindent Therefore, for any $E$,  $\lambda_v$ and  $0<\delta<\rho$, there is $\frac{\delta}{2}<y_0<\delta$ such that
\[
  \inf_{x\in[0,1]}\left |\lambda_v v(x+iy_0)-E\right |>\lambda_v \epsilon_0.
\]
Let
\begin{equation}\label{gh}
 M^a_{n-1}(x+iy_0,E,\omega)\left (\begin{array}
  {cc} 1\\ 0
\end{array}\right )=\left (\begin{array}
  {cc} g_{n-1}\\ h_{n-1}
\end{array}\right ).
\end{equation}
Then
\begin{eqnarray}\label{trgh}\left (\begin{array}
  {cc} g_n\\ h_n
\end{array}\right )&=&\left (\begin{array}
  {cc} \lambda_v v\left (x+iy+(n-1)\omega\right )-E &-\lambda_a\tilde a(x+iy+(n-1)\omega)\\ \lambda_a a(x+iy+n\omega)&0\\
\end{array}\right )\left (\begin{array}
  {cc} g_{n-1}\\ h_{n-1}
\end{array}\right )\\
&=&\left (\begin{array}
  {cc} \left ( \lambda_v v\left (x+iy+(n-1)\omega\right )-E\right )g_{n-1}-\lambda_a\tilde a(x+iy+(n-1)\omega)h_{n-1}\\ \lambda_a a(x+iy+n\omega)g_{n-1}
\end{array}\right ).\nonumber\end{eqnarray}
Set $\lambda_3=\lambda_3(v,\lambda_a,a)=2\lambda_a\|a\|_{L^{\infty}(\mathbb{T})}\epsilon^{-1}_0$. If $\lambda_v>\lambda_3$, then for any $E\in\mathscr{E}$, it implies
\[
  \inf_{x\in[0,1]}\left |\lambda_v v(x+iy_0)-E\right |>\lambda_v\epsilon_0>2\lambda_a\|a\|_{L^{\infty}(\mathbb{T})}.
\]
Now we use the induction to show that
\begin{equation}\nonumber
 |g_n|\ge |h_n|, \ \ \  \mbox{and}\ \ \  |g_n|\ge (\lambda_v\epsilon_0-\lambda_a\|a\|_{L^{\infty}(\mathbb{T})})|g_{n-1}|\ge (\lambda_v\epsilon_0-\lambda_a\|a\|_{L^{\infty}(\mathbb{T})})^n,\ \ n\ge 1.
\end{equation}
Due to (\ref{gh}) and (\ref{trgh}), it has that $g_0=1,\  h_0=0$  and
\[|g_1|=\left|\lambda_v v\left (x+iy\right )-E\right|>\lambda_v\epsilon_0,\ \ |h_1|=|\lambda_a a(x+n\omega)|\le\lambda_a\|a\|_{L^{\infty}(\mathbb{T})}.\]
Let $|g_t|\ge |h_t|$ and  $|g_t|>(\lambda\epsilon -\lambda_a\|a\|_{L^{\infty}(\mathbb{T})})|g_{t-1}|>(\lambda\epsilon_0-\lambda_a\|a\|_{L^{\infty}(\mathbb{T})})^t$. Then, we finish this induction by
\[
  |g_{t+1}|\ge \left|\left ( \lambda_v v\left (x+iy+t\omega\right )-E\right )g_{t}\right|-\left|\lambda_a\tilde a(x+iy+t\omega)h_{t}\right|>\left (\lambda_v\epsilon-\lambda_a\|a\|_{L^{\infty}(\mathbb{T})}\right )^{t+1},\]
and \[
  |h_{t+1}|\leq |\lambda_a a(x+iy+n\omega)g_t|<\lambda_a\|a\|_{L^{\infty}(\mathbb{T})}|g_t|\leq |g_{t+1}|.
\]
Therefore, we  have
\begin{equation}\nonumber
 \|M^a_n(x+iy_0,E,\omega)\|\ge \left|\left<M^a_n(x+iy_0,E,\omega)\left(\begin{array}{cc}
1\\ 0\\
\end{array}\right),\left(\begin{array}{cc}
1\\ 0\\
\end{array}\right)\right>\right|=|g_n|>(\lambda_v\epsilon_0-\lambda_a\|a\|_{L^{\infty}(\mathbb{T})})^n\ge \left(\frac{1}{2}\lambda_v \epsilon_0\right)^n.
\end{equation} It implies that
\begin{equation}\nonumber
  u_n^a(x+iy_0,E,\omega)=\frac{1}{n}\log\|M^a_n(x+iy_0,E,\omega)\|\ge \log\left(\frac{1}{2}\lambda_v\epsilon_0\right).
\end{equation}

Write $\mathbb{H}=\{z:\Im z>0\}$ for the upper half-plane and $\mathbb{H}_s$ for the strip $\{z=x+iy:0<y<\frac{\rho}{5}\}$. Then
denote by $\mu(z,\mathbb{E},\mathbb{H})$ the harmonic measure of $\mathbb{E}$ at $z\in\mathbb{H}$ and $\mu_s(iy_0,\mathbb{E}_s,\mathbb{H}_s)$ the harmonic measure of $\mathbb{E}_s$ at $iy_0\in\mathbb{H}_s$, where $\mathbb{E}\subset\partial \mathbb{H}=\mathbb{R}$ and $\mathbb{E}_s\subset \partial \mathbb{H}_s=\mathbb{R}\bigcup [y=\frac{\rho}{5}]$. Note that $\psi(z)=\exp\left (\frac{5\pi}{\rho}z\right )$ is a conformal map from $\mathbb{H}_s$ onto $\mathbb{H}$. Due to \cite{GM}, we have
\[
  \mu_s(iy_0,\mathbb{E}_s,\mathbb{H}_s)\equiv \mu(\psi(iy_0),\psi(\mathbb{E}_s),\mathbb{H}),
\]
and
\[
  \mu(z=x+iy,\mathbb{E},\mathbb{H})=\int_{\mathbb{E}}\frac{y}{(t-x)^2+y^2}\frac{dt}{\pi}.
\]
Thus
\[\mu_s[y=\frac{\rho}{5}]=\frac{5\pi y_0}{\pi\rho}<\frac{5\delta}{\rho}.\]
By the subharmonicity and (\ref{upb}), it yields that if $\lambda_v>\max(\lambda_1,\lambda_3)$, then
\begin{eqnarray}
  \log \left(\frac{1}{2}\lambda_v \epsilon_0\right)<u^a_n(iy_0,E,\omega)& \leq &\int_{[y=0]\bigcup [y=\frac{\rho}{5}]}u^a_n(z,E,\omega)\mu_s(dz)\nonumber\\
  &=&\int_{y=0}u^a_n(x,E,\omega)\mu_s(dx)+\int_{y=\frac{\rho}{5}}u^a_n(x+iy,E,\omega)\mu_s(dx)\nonumber\\
  &\leq &\int_{\mathbb{R}}u^a_n(x,E,\omega)\mu_s(dx)+\frac{5\delta}{\rho}\left [\sup_{y=\frac{\rho}{5}}u^a_n(x+iy,E,\omega)\right ]\nonumber\\
  &\leq &\int_{\mathbb{R}}u^a_n(x,E,\omega)\mu_s(dx)+\frac{5\delta}{\rho}\log\left(5\lambda_v\|v\|_{L^{\infty}(\Omega)}\right)\nonumber.
\end{eqnarray}
So, if $\delta<\frac{\gamma\rho}{10}$ and $\lambda_v>\lambda_4:=5\|v\|_{L^{\infty}(\Omega)}\left(\frac{2}{\epsilon_0}\right)^{\frac{2}{\gamma}}$, then
\begin{equation}\label{es-un}
  \int_{\mathbb{R}}u_n(x,E,\omega)\mu_s(dx)\ge \log \left(\frac{1}{2}\lambda_v \epsilon_0\right)-\frac{5\delta}{\rho}\log\left(5\lambda_v\|v\|_{L^{\infty}(\Omega)}\right)>\left (1-\gamma\right )\log \lambda_v.
\end{equation}
Set
\[u^h_n(x)=u^a_n(x+h),\ \ h\in \mathbb{T}.\]
Then, due to Lemma \ref{BG}, it is obvious that (\ref{es-un}) also holds for $u^h_n(x)$. So, for any $h\in\mathbb{T}$, it has
\[\int_{\mathbb{R}}u^a_n(x+h)\mu_s(dx)> \left (1-\gamma\right )\log \lambda_v.\]
Integrating in $h\in \mathbb{T}$ implies that
\begin{eqnarray}
  L^a_n(E,\omega)=\int_0^1u^a_n(x+h,E,\omega)dh&\ge & \left (\int_{\mathbb{R}}\mu_s(dx)\right )\times\left ( \int_0^1u^a_n(x+h,E,\omega)dh\right )\nonumber \\
  &=&\int_0^1\int_{\mathbb{R}}u_n(x+h,E,\omega)\mu_s(dx)dxdh \nonumber\\
  &>&\left (1-\gamma\right )\log \lambda_v,\ \ \forall n\ge 0.\nonumber
\end{eqnarray}
Thus,  combining it with (\ref{suplea}) and (\ref{supd}), we  finish the proof of Theorem \ref{posile} with  $n\to +\infty$ and
\begin{equation}\label{lambda0}
  \lambda_v>\lambda_p:=\max\left(5\|v\|_{L^{\infty}(\mathbb{T})}^{\frac{1}{\gamma}},2\lambda_a\|a\|_{L^{\infty}(\mathbb{T})},
  \frac{\left(\lambda_a\|a\|_{L^{\infty}(\Omega)}\right)^{\frac{2}{\gamma}}}{\|v\|_{L^{\infty}(\Omega)}}\right)\left(\frac{2}{\epsilon_0}\right)^{\frac{2}{\gamma}}
  >\max(\lambda_1,\lambda_2,\lambda_3,\lambda_4,\lambda_D).
\end{equation}
\begin{remark}
  \label{schpole}
  For the Schr\"odinger operators,
  due to Lemma \ref{BG}, $\|v\|_{L^{\infty}(\Omega)}>\epsilon_0$. Therefore, there exists $\lambda_p^s:=\left(\frac{20\|v\|_{L^{\infty}(\Omega)}}{\epsilon_0^2}\right)^{\frac{1}{\gamma}}$ such that if $\lambda_s>\lambda_p^s$, then  $L^s(E,\omega)$ is positive for any irrational $\omega$ and  $E\in\mathcal{E}_s$ as follows:
  \[(1-\gamma)\log\lambda_s<L_n^s(E,\omega)<(1+\gamma)\log\lambda_s,\ \ \forall n\ge 1,\]
  where $\mathcal{E}_s:=\left[-2-\lambda\|v\|_{L^{\infty}(\Omega)},2+\lambda\|v\|_{L^{\infty}(\Omega)}\right]$. What's more, for any irrational $\omega$ , $E\in\mathcal{E}_s$ and  $x\in\mathbb{T}$,
  \[u_n^s(x,E,\omega)\leq M^s_0\leq (1+\gamma)\log\lambda.\]
\end{remark}
~\\

On the other hand, if we choose $
   \Omega_1$  in Lemma \ref{lem:riesz} as  (\ref{Omega1}) and $\delta=\frac{\rho}{2}$ in Lemma \ref{BG}, then
\[\sup\limits_{E\in\mathscr{E},z\in\Omega_1}u^a_n(z,E,\omega)\ge \log\left(\frac{1}{2}\lambda_v\epsilon_0\right),\ \ \lambda_v>\lambda_3.\]
Combining it with (\ref{21002}) and (\ref{upb}), we  have the following lemma:
\begin{lemma}\label{muomu}
 There exist $\lambda_0=\lambda_0(v,\lambda_a,a):=\max(\lambda_1,\lambda_3)$ and $C_v=C_v(v)=C(\Omega,\Omega_1)\log\frac{10\|v\|_{L^{\infty}(\Omega)}}{\epsilon_0}$ such that for $\lambda_v>\lambda_0$,$$\mu_u(\Omega_1)\leq C_v,$$
 where $\mu_u$ is the unique measure for $u^a_n(z,E,\omega)$ in Lemma \ref{lem:riesz}.
\end{lemma}
\noindent Thus, Theorem \ref{erg} can be applied to $u_n^a(x,E,\omega)$ as follows:
\begin{lemma}\label{erga}
There exists $c_v=c_v(v):=\frac{1}{2CC_v}$ such that if $\beta<c_v\delta$ and $\lambda_v>\lambda_0$, then for any positive $k$ and $n$,
\[\mes\left (\left \{x\in \mathbb{T}:|\frac{1}{k}\sum_{j=1}^k u^a_n(x+j\omega)-L^a_n(E,\omega)|>\delta \right \}\right )<\exp(-\bar c_v\delta k),\]
where $\bar c_v=\frac{c}{C_v}$.
\end{lemma}
\begin{remark}\label{ergs}
  For the Schr\"odinger operators($\lambda_aa\equiv1$), we have that if $\beta<c_s\delta$ and $\lambda_s>\lambda_0^s(v):=2\epsilon_0^{-1}$, then for any positive $k$ and $n$,
  \[\mes\left (\left \{x\in \mathbb{T}:|\frac{1}{k}\sum_{j=1}^k u^s_n(x+j\omega)-L^s_n(E,\omega)|>\delta \right \}\right )<\exp(-\bar c_v\delta k).\]
\end{remark}

Similar computations show that for any $\lambda_a\not=0$,
\[\mu_d(\Omega_1)\leq C(\Omega,\Omega_1)\log\frac{\|a\|_{L^{\infty}(\Omega)}}{\|a\|_{L^{\infty}(\Omega_1)}}:=C_a,\]
where $\mu_d$ is the unique measure in Lemma \ref{lem:riesz} for $d(z,\omega)$ defined in (\ref{defd}). Correspondingly, the following two strong Birkhoff Ergodic Theorems  both hold:
\begin{lemma}\label{ldtd}
There exists $c_a=c_a(a)=\frac{1}{2CC_a}$  such that if $\beta<c_a\delta$, then for any positive $k$,
\[\mes\left (\left \{x\in \mathbb{T}:\left|\frac{1}{k}\sum_{j=1}^k \log|a(x+j\omega)|-D\right|>\delta \right \}\right )<\exp(-\bar c_a\delta k),\]
and
\[\mes\left (\left \{x\in \mathbb{T}:\left|\frac{1}{k}\sum_{j=1}^k d(x+j\omega,\omega)-2D\right|>\delta \right \}\right )<\exp(-\bar c_a\delta k),\]
where $\bar c_a=\frac{c}{C_a}$.
\end{lemma}
\begin{remark}\label{ergu}
By (\ref{re1}), Lemma \ref{erga} and Lemma \ref{ldtd}, it implies that if $\beta<\min(c_v\delta,c_a\delta)$ and $\lambda_v>\lambda_0$, then
\[\mes\left (\left \{x\in \mathbb{T}:|\frac{1}{k}\sum_{j=1}^k u^u_n(x+j\omega)-L_n(E,\omega)|>2\delta \right \}\right )<\exp(-\bar c_v\delta k)+\exp(-\bar c_a\delta k).\]Set $C_{v,a}=\max(\log\frac{10\|v\|_{L^{\infty}(\Omega)}}{\epsilon_0},\log\frac{\|a\|_{L^{\infty}(\Omega)}}{\|a\|_{L^{\infty}(\Omega_1)}})$. Then, there exist $c_{v,a}=\frac{1}{2CC(\Omega,\Omega_1)C_{v,a}}$ and $\bar c_{v,a}=\frac{c}{C(\Omega,\Omega_1)C_{v,a}}$ such that for any $\delta>0$, if $\beta<c_{v,a}\delta$ and $\lambda_v>\lambda_0$, then
\begin{equation}\label{ergeu}
  \mes\left (\left \{x\in \mathbb{T}:|\frac{1}{k}\sum_{j=1}^k u^u_n(x+j\omega)-L_n(E,\omega)|>\delta \right \}\right )<\exp(-\bar c_{v,a}\delta k).
\end{equation}
And to reduce the plaguing of too many symbols, we can use $c_{v,a}$ and $\bar c_{v,a}$ instead of $c_v$ and $\bar c_v$ in Lemma \ref{erga}, and of $c_a$ and $\bar c_a$ in Lemma \ref{ldtd}, as $c_{v,a}=\min(c_v,c_a)$ and $\bar c_{v,a}=\min(\bar c_v,\bar c_a)$.
\end{remark}

\section{The Proof of Theorem \ref{ldt}}

Define
$$\mathbb{X}_m=\left\{x\in\mathbb{T}:\left|\frac{1}{m}\sum_{j=0}^{m-1}d(x+j\omega)-2D\right|>\frac{k}{m}D\right\}.$$By Lemma \ref{ldtd}, we have that if $\beta<c_{v,a}\delta$, then  $\mes\left(\mathbb{X}_m\right)=\exp(-\bar c_{v,a}\delta k)$ for any $ 1\le m\le k$. It implies that
\begin{equation}
  \mes\left(\left\{x\in\mathbb{T}:\left|\sum_{j=0}^{k-1}\frac{k-j}{k}d(x+j\omega)-(k+1)D\right|>k\delta\right\}\right)<k\exp(-\bar c_{v,a}\delta k).
\end{equation}
Corollary 2.3 in [T] proved that
\begin{equation}
  -\frac{2Mk}{n}+\sum_{j=0}^{k-1}\frac{k-j}{nk}d(x+j\omega)
\leq u^a_n(x,E,\omega)-\frac{1}{k}\sum_{j=1}^ku^a_n(x+k\omega,E,\omega)\leq \frac{2Mk}{n}-\sum_{j=0}^{k-1}\frac{k-j}{nk}d(x+(n+j-1)\omega).
\end{equation}
So, we define
$$\mathbb{Y}_{-}=\left\{x\in\mathbb{T}: -\frac{2Mk}{n}+\sum_{j=0}^{k-1}\frac{k-j}{nk}d(x+j\omega)<-\delta\right\}$$
and let $k=C_1\delta n$ and $4C_1M\leq 1$. Then
\[\mathbb{Y}_{-}\subset \left\{x\in\mathbb{T}: \sum_{j=0}^{k-1}\frac{k-j}{k}d(x+j\omega)<-\frac{\delta n}{2}=-\frac{k}{2C_1}\right\}.\]
Assume $6C_1|D|\leq 1$ to make $\frac{1}{2C_1}+D=C_2>2|D|\ge 0$ and
\[\frac{k}{2C}+(k+1)D=C_2k+D\ge |D|k.\]
It implies that if $\beta<c_{v,a}|D|$, then
\begin{eqnarray*}
  \mes(\mathbb{Y})_{-}&< &  \mes\left(\left\{x\in\mathbb{T}: \sum_{j=0}^{k-1}\frac{k-j}{k}d(x+j\omega)-(k+1)D  <-k|D| \right\} \right)<k\exp(-\bar c_{v,a}|D| k)\\
  &=&C_1\delta n\exp(-\bar c_{v,a}C_1|D| \delta n).
\end{eqnarray*}Because $y\exp(-\zeta y)\leq \zeta^{-1}$for any $y,\zeta>0$,so
\begin{eqnarray*}
  C_1\delta n\exp(-\bar c_{v,a}C_1|D| \delta n)&=&C_1\delta n\exp(-\frac{\bar c_{v,a}|D|}{2}C_1 \delta n)\exp(-\frac{\bar c_{v,a}|D|}{2}C_1 \delta n)\leq \frac{2}{\bar c_{v,a}|D|}\exp(-\frac{\bar c_{v,a}|D|}{2}C_1 \delta n)\\
  &=&\exp(-\bar c_1\delta n), \forall n\ge  n_1,
\end{eqnarray*}
where $\bar c_1=\bar c_1(v,a,|D|, M_0)=\frac{\bar c_{v,a}|D|C_1}{4}<\bar c_{v,a}$ and $n_1=n_1(a,v,|D|,M_0,\delta)$ satisfying
 \begin{equation}\label{n1}
   n_1=\frac{-4\log\frac{\bar c_{v,a}|D|}{2}}{\bar c_{v,a}|D|C_1\delta}.
 \end{equation}Similar calculations show that
\[\mes\left(\mathbb{Y}_{+}\right)=\mes\left(\left\{x\in\mathbb{T}: \frac{2Mk}{n}-\sum_{j=0}^{k-1}\frac{k-j}{nk}d(x+(n+j-1)\omega)>\delta\right\}\right)<\exp(-\bar c_1\delta n).\]
Therefore, we have the deviation estimation  as follows:
\begin{lemma}\label{deu}For any $\delta>0$, if $\beta<c_{v,a}\min(\delta,|D|)$, then
\[\mes\left(\left\{x\in\mathbb{T}: \left|u^a_n(x,E,\omega)-\frac{1}{k}\sum_{j=1}^ku^a_n(x+k\omega,E,\omega)\right|>\delta\right\}\right)<2\exp(-\bar c_1\delta n),\ \ \forall n\ge n_1.\]
\end{lemma}
\noindent Combining it with Lemma \ref{erga}, we have the following LDT for $u_n^a(x,E,\omega)$:
\begin{lemma}\label{ldta}
  For any $\delta>0$, if $\beta<c_{v,a}\min(\delta,|D|)$ and $\lambda_v>\lambda_0$, then
\[\mes\left (\left \{x\in \mathbb{T}:| u^a_n(x,E,\omega)-L^a_n(E,\omega)|>\frac{3\delta}{4} \right \}\right )<\exp(-\bar c_{a,v}C_1\left(\frac{\delta}{4}\right)^2 n)+2\exp(-\bar c_1\frac{\delta}{4} n),\ \ \forall n\ge  n_1.\]
\end{lemma}
\begin{remark}\label{ldts}
  For the Schr\"odinger operators, $d(x)\equiv 1$ and
  \[\left|u_n^s(x,E,\omega)-\frac{1}{k}u_n^s(x+k\omega,E,\omega)\right|\leq \frac{2M_0^sk}{n}.\]
  Then, due to Remark \ref{ergs} and the setting $k=\frac{\delta n}{4M_0^s}$, we have if $\beta<c_v\delta$ and $\lambda_s>\lambda_0^s$, then there exists $\bar c_s=\bar c_s(\lambda ,v):= \frac{1}{8M_0^s}\bar c_v$ such that for any positive  $n$,
  \[\mes\left (\left \{x\in \mathbb{T}:| u^s_n(x,E,\omega)-L^s_n(E,\omega)|>\delta \right \}\right )<\exp\left(-\bar c_v\frac{\delta}{2}\frac{\delta n}{4M_0^s}\right)=\exp(-\bar c_s \delta^2 n).\]
\end{remark}
\begin{proof}[Proof of Theorem \ref{ldt}]
The theorem is obtained directly by the setting of $\bar c_1$, (\ref{de2}), Lemma \ref{ldtd} and Lemma \ref{ldta}.
\end{proof}~\\

With the similar process by changing (\ref{de2}) to (\ref{re1}), we have the following LDT for $u_n^u(x,E,\omega)$, which will be applied
to satisfy the assumption (\ref{eq:detsmall}) in  the Avalanche Principle:
 \begin{lemma}\label{ldtu}
  For any $\delta>0$ and  $E\in\mathscr{E}$, if $\beta<c_{v,a}\min(\delta,|D|)$ and $\lambda_v>\lambda_0$, then
\[\mes\left (\left \{x\in \mathbb{T}:| u^u_n(x,E,\omega)-L_n(E,\omega)|>\delta \right \}\right )<\exp(-\bar c_{a,v}C_1\left(\frac{\delta}{4}\right)^2 n)+3\exp(-\bar c_1\frac{\delta}{4} n), \forall n\ge  n_1.\]
\end{lemma}
\begin{remark}
  If $\delta <\delta_0= \frac{8\bar c_1}{\bar c_{av}C_1}$, then
  \begin{equation}\label{ldteu}\mes\left (\left \{x\in \mathbb{T}:| u^u_n(x,E,\omega)-L_n(E,\omega)|>\delta \right \}\right )<\exp(- \bar c_u\delta^2n),\ \ \forall n\ge  n_1 ,\end{equation}
  where $\bar c_u=\frac{\bar c_{a,v}C_1}{20}$.
\end{remark}

\section{Applications of Avalanche Principle and the Positive Lyapunov Exponents on an Interval}
Avalanche Principle is the following:
\begin{prop}[Avalanche Principle]
\label{prop:AP} Let $A_1,\ldots,A_n$ be a sequence of  $2\times
2$--matrices whose determinants satisfy
\begin{equation}
\label{eq:detsmall} \max\limits_{1\le j\le n}|\det A_j|\le 1.
\end{equation}
Suppose that \be
&&\min_{1\le j\le n}\|A_j\|\ge\gamma>n\mbox{\ \ \ and}\label{large}\\
   &&\max_{1\le j<n}[\log\|A_{j+1}\|+\log\|A_j\|-\log\|A_{j+1}A_{j}\|]<\frac12\log\gamma\label{diff}.
\ee Then
\begin{equation}
\Bigl|\log\|A_n\cdot\ldots\cdot A_1\|+\sum_{j=2}^{n-1}
\log\|A_j\|-\sum_{j=1}^{n-1}\log\|A_{j+1}A_{j}\|\Bigr| <
C\frac{n}{\gamma} \label{eq:AP}
\end{equation}
with some absolute constant $C$.
\end{prop}

\begin{lemma}\label{62001}
Assume $L_n(E,\omega)>0$, $\delta=\min\left(1,\delta_0,\frac{1}{15}L_n(E,\omega)\right)$, $L_{2n}(E,\omega)>\frac{9}{10}L_{n}(E,\omega)$, $\beta<c_{v,a}\min\left(\frac{L_n(E,\omega)}{15},|D|\right)$ and $\lambda_v>\lambda_0$.  Let $N=m
n$, $m\in \mathbb{N}$ and
$\exp(\frac{\bar c_u}{3}\delta^2 n) \leq m\leq
\exp(\frac{\bar c_u}{3}\delta^2 n)+1 $. There exists $n_2=n_2(\lambda_a,a,\lambda_v,v,\delta)$ such that for any $n\ge  n_2$,  \be
|L_{N}(E,\omega)+L_{n}(E,\omega)-2L_{2n}(E,\omega)|\leq \exp(-\frac{\bar c_u}{20}\delta^2 n).\ee
\end{lemma}
\begin{proof}
By (\ref{ldteu}), we have, for $0\leq j\leq m-1$ and $\forall x\in\mathbb{G}$,
\[|u_n^u(x+jn\omega,E,\omega)-L_{n}(E,\omega)|<\frac{L_n(E,\omega)}{15},\]
\[|u_{2n}^u(x+jn\omega,E,\omega)-L_{2n}(E,\omega)|<\frac{L_n(E,\omega)}{15},\]
with
\[\mes(\mathbb{T}\backslash \mathbb{G})\leq 2m \times\exp\left(-\bar c_u\delta^2 n\right)<2\exp(-\frac{2\bar c_u}{3}\delta^2n).\]
Thus, when $x\in\mathbb{G}$,
\[\|M_{n}^u(x+jn\omega,E,\omega)\|>\exp(\frac{14}{15}nL_{n}(E,\omega)),\]
and
\begin{eqnarray}\nonumber
&&\bigg |\log\|M_n^u(x+jn\omega,E,\omega)\|+\log\|M_{n}^u(x+(j+1)n\omega,E,\omega)\|-\log\|M_n^u(x+jn\omega,E,\omega)M_{n}^u(x+(j+1)n\omega,E,\omega)\|\bigg |\\
&\ &\ \ \ \ \ \ \ \ <4n\frac{L_n(E,\omega)}{100}+2n|L_{n}(E,\omega)-L_{2n}(E,\omega)|
<\frac{7}{15}nL_{n}(E,\omega).\nonumber\end{eqnarray}Therefore, Avalanche Principle applies for
$\gamma=\exp(\frac{14}{15}nL_n(E))$. Integrating over $\mathbb{G}$, we
obtain
\begin{eqnarray}\label{60201}
&&\left|\int_{\mathbb{G}}u_{N}^u(x,E,\omega)dx+\frac{1}{m}\int_{\mathbb{G}}\sum_{j=2}^{m-1}u^u_{n}(x+(j-1)n\omega,E,\omega)dx-\frac{2}{m}\int_{\mathbb{G}}\sum_{j=1}^{m-1}
u^u_{2n}(x+(j-1)n\omega,E,\omega)dx\right|\\
&\leq&
C\frac{m}{N}\exp(-\frac{14}{15}nL_n(E,\omega)).\nonumber\end{eqnarray}  We want to replace  integration over $\mathbb{G}$
by integration over $\mathbb{T}$. By the Cauchy-Schwartz inequality, it has for any $E $, $\omega$ and $n$,
\[\left|\int_{\mathbb{T}\backslash\mathbb{G}}u_n^u(x,E,\omega)dx\right|\leq \|u_n^u(\cdot,E,\omega)\|_{L^2(\mathbb{T})}\bigg(\mes\left(\mathbb{T}\backslash\mathbb{G}\right)\bigg)^{\frac{1}{2}}<\|u_n^u(\cdot,E,\omega)\|_{L^2(\mathbb{T})}\exp(-\frac{\bar c_u}{3}\delta^2n).\]
Thus
\begin{eqnarray*}&&\left|\int_{\mathbb{T}\backslash\mathbb{G}}u_N^u(x,E,\omega)dx+\frac{1}{m}\int_{\mathbb{T}\backslash\mathbb{G}}\sum_{j=2}^{m-1}u_n^u(x+(j-1)m\omega,E,\omega)dx-\frac{2}{m}\int_{\mathbb{T}\backslash\mathbb{G}}\sum_{j=1}^{m-1}
u^u_{2n}(x+(j-1)N\omega,E,\omega)dx\right|\\&\leq&
4\|u_n^u(\cdot,E,\omega)\|_{L^2(\mathbb{T})}\exp(-\frac{\bar c_u}{3}\delta^2n).\nonumber
\end{eqnarray*}
Combining it with (\ref{60201}), we have
\begin{equation}
|L_{N}(E,\omega)+\frac{m-2}{m}L_n(E,\omega)-\frac{2(m-1)}{m}L_{2n}(E,\omega)| \leq
4\|u_n^u(\cdot,E,\omega)\|_{L^2(\mathbb{T})}\exp(-\frac{\bar c_u}{3}\delta^2n)+C\frac{m}{N}\exp(-\frac{14}{15}nL_n(E,\omega))\nonumber.\end{equation}
Thus, if $\exp(\frac{\bar c_u}{3}\delta^2 n) \leq m\leq
\exp(\frac{\bar c_u}{3}\delta^2 n)+1 $, then
\begin{eqnarray*}&& |L_{N}(E,\omega)+L_n(E,\omega)-2L_{2n}(E,\omega)|\\&\leq &
4\|u_n^u(\cdot,E,\omega)\|_{L^2(\mathbb{T})}\exp(-\frac{\bar c_u}{3}\delta^2n)+C\frac{m}{N}\exp(-\frac{1}{2}nL_n(E,\omega))+\frac{2}{m}|L_n(E,\omega)-L_{2n}(E,\omega)|\\
&<& 4\sup_{E\mathcal{E}}\|M^u(\cdot,E,\omega)\|_{L^2(\mathbb{T})}\exp(-\frac{\bar c_u}{3}\delta^2n)+C\frac{m}{N}\exp(-\frac{14}{15}nL_n(E,\omega))+\frac{1}{5m}L_n(E,\omega)\nonumber\\
&\leq
&\exp(-\frac{\bar c_u}{4}\delta^2n),\ \ \ \forall n\ge n_2,\nonumber\end{eqnarray*}where
\begin{equation}
  \label{n2}n_2=n_2(\lambda_a,a,\lambda_v,v,\delta)=\frac{12\log\left(5\sup_{E\in\mathcal{E}}\|M^u(\cdot,E,\omega)\|_{L^2(\mathbb{T})}\right)}{\bar c_u\delta^2}.
\end{equation}
\end{proof}

\begin{lemma}\label{62002}
Assume $L_n(E,\omega)>0$, $\delta=\min\left(1,\delta_0,\frac{1}{15}L_n(E,\omega)\right)$, $L_{2n}(E,\omega)>\frac{9}{10}L_{n}(E,\omega)$, $\beta<c_{v,a}\min\left(\frac{L_n(E,\omega)}{15},|D|\right)$ and $\lambda_v>\lambda_0$. Then, for  any $n\ge  n_2$,
 \be |L(E,\omega)+L_{n}(E,\omega)-2L_{2n}(E,\omega)|< \exp\left(-\frac{\bar c_u}{5}\delta^2
 n\right).\ee
\end{lemma}
\begin{proof}
By lemma \ref{62001} for $N_0=n$, $N_1=mN_0$ and
$\exp(\frac{\bar c_u}{3}\delta^2 N_0)\leq m<
\exp(\frac{\bar c_u}{3}\delta^2 N_0)+1$, we have
\begin{equation}\label{54001}|L_{N_1}(E,\omega)+L_{N_0}(E,\omega)-2L_{2N_0}(E,\omega)|<
\exp(-\frac{\bar c_u}{4}\delta^2 N_0),\end{equation} and
\[|L_{2N_1}(E,\omega)+L_{N_0}(E,\omega)-2L_{2N_0}(E,\omega)|<
\exp(-\frac{\bar c_u}{4}\delta^2 N_0).\]In particular
\[|L_{N_1}(E,\omega)-L_{2N_1}(E,\omega)|< 2\exp(-\frac{\bar c_u}{4}\delta^2 N_0).\]
Since $0\leq L_{N_0}(E,\omega)-L_{2N_0}(E,\omega)<\frac{1}{10}L_{N_0}(E,\omega)$ and (\ref{54001}), we
obtain that
\[L_{N_1}(E,\omega)>L_{N_0}(E,\omega)-2\big(L_{N_0}(E,\omega)-L_{2N_0}(E,\omega)\big)-\exp(-\frac{\bar c_u}{4}\delta^2 N_0)>\frac{4}{5}L_{N_0}(E,\omega)-\exp(-\frac{\bar c_u}{4}\delta^2 N_0)>79\delta,\]
and
\[|L_{N_1}(E,\omega)-L_{2N_1}(E,\omega)|\leq 2\exp(-\frac{\bar c_u}{20}\delta^2 N_0)<2\delta<\frac{2}{79}L_{N_1}(E,\omega)<\frac{1}{10}L_{N_1}(E,\omega).\]
Set $\delta'=\frac{1}{2}\delta$. Then $L_{N_1}(E,\omega)>15 \delta'$, and
Lemma \ref{62001} applies for $N_2=m_1N_1$ and
$\exp(\frac{\bar c_u}{3}\delta'^2 N_1)\leq m_1<
\exp(\frac{\bar c_u}{3}\delta'^2 N_1)+1$. Therefore,
\[|L_{N_2}(E,\omega)+L_{N_1}(E,\omega)-2L_{2N_1}(E,\omega)|\leq
\exp(-\frac{\bar c_u}{4}\delta'^2 N_1),\]
\[L_{N_2}(E,\omega)>L_{N_1}(E,\omega)-2|L_{N_1}(E,\omega)-L_{2N_1}(E,\omega)|-\exp(-\frac{\bar c_u}{4}\delta'^2 N_1)>\frac{4}{5}L_{N_0}(E,\omega)-6\exp(-\frac{\bar c_u}{4}\delta^2
N_0)>79\delta>100\delta',\]
\[|L_{2N_2}(E,\omega)+L_{N_1}(E,\omega)-2L_{2N_1}(E,\omega)|\leq
\exp(-\frac{\bar c_u}{4}\delta'^2 N_1),\]and
\[|L_{N_2}(E,\omega)-L_{2N_2}(E,\omega)|<2\exp(-\frac{\bar c_u}{4}\delta'^2 N_1).\]
Since $N_1>8N_0$, we have
\[\exp(-\frac{\bar c_u}{4}\delta'^2N_1)=\exp(-\frac{\bar c_u}{4}\frac{\delta^2}{4}N_1)<(\exp(-\frac{\bar c_u}{4}\delta^2
N_0))^2<(\frac{\delta}{12})^2.\] This implies in particular that
\[|L_{N_2}(E,\omega)-L_{2N_2}(E,\omega)|<2\exp(-\frac{\bar c_u}{4}\delta'^2
N_1)<2\delta<\frac{1}{10}L_{N_2}(E,\omega).\] Then Lemma \ref{62001}
applies for $N_3=m_2N_2$ and $\exp(\frac{\bar c_u}{3}\delta'^2
N_2)\leq m_2<\exp(\frac{\bar c_u}{3}\delta'^2 N_2)+1$. E.T.C.. We
obtain $N_{i+1}=m_iN_i$ and $\exp(\frac{\bar c_u}{3}\delta'^2
N_i)\leq m_i<\exp(\frac{\bar c_u}{3}\delta'^2 N_i)+1$. Then
\[|L_{N_{i+1}}(E,\omega)+L_{N_i}(E,\omega)-2L_{2N_i}(E,\omega)|\leq
\exp(-\frac{\bar c_u}{4}\delta'^2 N_i),\]
\[L_{N_{i+1}}(E,\omega)>L_{N_i}(E,\omega)-2|L_{N_i}(E,\omega)-L_{2N_i}(E,\omega)|-\exp(-\frac{\bar c_u}{4}\delta'^2 N_i)>\frac{4}{5}L_{N_0}(E,\omega)-\sum_{j=1}^{i}(\frac{1}{2})^j\delta\geq 79\delta>50\delta=100\delta',\]
\[|L_{2N_{i+1}}(E,\omega)+L_{N_i}(E,\omega)-2L_{2N_i}(E,\omega)|\leq
\exp(-\frac{\bar c_u}{4}\delta'^2 N_i),\]
\[|L_{N_{i+1}}(E,\omega)-L_{2N_{i+1}}(E,\omega)|<2\exp(-\frac{\bar c_u}{4}\delta'^2 N_i),\]
\[4\exp(-\frac{\bar c_u}{4}\delta'^2 N_i)<(\frac{1}{2})^{i+1}\delta,\]and
\[|L_{N_{i+1}}(E,\omega)-L_{2N_{i+1}}(E,\omega)|<2\delta<\frac{1}{10}L_{N_{i+1}}(E,\omega).\]
Moreover,
\begin{eqnarray}\nonumber
|L_{N_{i+1}}(E,\omega)-L_{N_{i}}(E,\omega)|&\leq&
|L_{N_{i+1}}(E,\omega)+L_{N_i}(E,\omega)-2L_{2N_i}(E,\omega)|+2|L_{N_i}(E,\omega)-L_{2N_i}(E,\omega)|\\
&<&\exp(-\frac{\bar c_u}{4}\delta'^2 N_i)+4\exp(-\frac{\bar c_u}{4}\delta'^2
N_{i-1})<5\exp(-\frac{\bar c_u}{4}\delta'^2 N_{i-1}),\ \ i\geq 2\nonumber,
\end{eqnarray}and
\[|L_{N_2}(E,\omega)-L_{N_1}(E,\omega)|<5\exp(-\frac{\bar c_u}{4}\delta^2 N_0).\]
Since
$L_{N_i}\to L(E,\omega)$ with $i\to \infty$, we have
\begin{eqnarray}\nonumber
&&|L(E,\omega)+L_{N_0}(E,\omega)-2L_{2N_0}(E,\omega)|\\
&=&\big|\sum_{i\geq
1}\big(L_{N_{i+1}}(E,\omega)-L_{N_{i}}(E,\omega)\big)+L_{N_1}(E,\omega)+L_{N_0}(E,\omega)-2L_{2N_0}(E,\omega)\big|\nonumber\\
&\leq &\sum_{s\geq 1}|L_{N_{s+1}}(E,\omega)-L_{N_{s}}(E,\omega)|+|L_{N_1}(E,\omega)+L_{N_0}(E,\omega)-2L_{2N_0}(E,\omega)|\nonumber\\
&=&\sum_{s\geq
2}|L_{N_{s+1}}(E,\omega)-L_{N_{s}}(E,\omega)|+|L_{N_2}(E,\omega)-L_{N_1}(E,\omega)|+|L_{N_1}(E,\omega)+L_{N_0}(E,\omega)-2L_{2N_0}(E,\omega)|\nonumber\\
&<&\sum_{s\geq 2}5\exp(-\frac{\bar c_u}{4}\delta'^2
N_{i-1})+5\exp(-\frac{\bar c_u}{4}\delta^2
N_0)+\exp(-\frac{\bar c_u}{4}\delta^2 N_0)\nonumber\\
&<&\exp(-\frac{\bar c_u}{5}\delta^2 N_0).\nonumber\end{eqnarray}
\end{proof}

\begin{lemma}\label{62003}Assume $L(E_0,\omega_0)>0$. There exists
$n_3=n_3(\lambda_aa,\lambda_vv,L(E_0,\omega_0))$ such that for any $n\ge n_3$, if $|E-E_0|<r_E(n)=\frac{L(E_0,\omega_0)}{200n}\exp\left(-(n-1)M_0-2|D|n\right)$, then
\begin{equation}\label{disE}|L_n(E_0,\omega_0)-L_n(E,\omega_0)|\leq\frac{L(E_0,\omega_0)}{100}.\end{equation}
\end{lemma}
\begin{proof}
Note that
\begin{eqnarray}\nonumber
&&\bigg |\|M_n^a(x,E_0,\omega)\|-\|M_n^a(x,E,\omega)\|\bigg |
\leq \|M_n^a(x,E_0,\omega)-M_n^a(x,E,\omega)\|\\
&\leq&\sum_{j=0}^{n-1}\bigg(\|M^a(x+(n-1)\omega,E_0,\omega)\times\cdots\times
M^a(x+(j+1)\omega,E_0,\omega)\|\times \nonumber\\&\ &\ \ \ \
\|M^a(x+j\omega,E_0,\omega)-
M^a(x+j\omega,E,\omega)\|\times\|M^a(x+(j-1)\omega,E,\omega)\times\cdots\times
M^a(x,E,\omega)\|\bigg)
\nonumber\\
&\leq&ne^{(n-1)M_0}|E_0-E|\nonumber.
\end{eqnarray}
By (\ref{de2}), we have
\begin{eqnarray}\nonumber
\bigg|\|M_n^u(x,E_0,\omega)\|-\|M_n^u(x,E,\omega)\|\bigg
|&=&\exp\left(-\frac{1}{2}\sum_{j=0}^{n-1}d(x+j\omega,\omega)\right)\bigg
|\|M_n^a(x,E_0,\omega)\|-\|M_n^a(x,E,\omega)\|\bigg |\\
&\leq
&ne^{(n-1)M_0}\exp\left(-\frac{1}{2}\sum_{j=0}^{n-1}d(x+j\omega,\omega)\right)|E_0-E|.\nonumber\end{eqnarray}
Assume, for instance, that $\|M_n^u(x,E_0,\omega)\|\geq
\|M_n^u(x,E,\omega)\|$. Then
\begin{eqnarray}\label{60002}\left |\log\|M_n^u(x,E_0,\omega)\|-\log \|M_n^u(x,E,\omega)\| \right |&=&
\log
(1+\frac{\|M_n^u(x,E_0,\omega)\|-
\|M_n^u(x,E,\omega)\|}{\|M_n^u(x,E,\omega)\|})\\
&\leq &\left|\frac{\|M_n^u(x,E_0,\omega)\|-
\|M_n^u(x,E,\omega)\|}{\|M_n^u(x,E,\omega)\|}\right|\nonumber\\
&\leq&\left|\|M_n^u(x,E_0,\omega)\|-
\|M_n^u(x,E,\omega)\|\right|\nonumber\\
&\le &ne^{(n-1)M_0}\exp\left(-\frac{1}{2}\sum_{j=0}^{n-1}d(x+j\omega,\omega)\right)|E_0-E|.\nonumber\end{eqnarray}
Thus, due to Lemma \ref{ldtd}, we have that if $\beta\leq c_{v,a}|D|$, then there exists $\mathbb{B}_{D}$ satisfying $\mes(\mathbb{B}_{D})\leq \exp(-\bar c_{v,a}|D|n)$ such that when $x\not\in \mathbb{B}_{D}$,
\[\sum_{j=0}^{n-1}\frac{1}{2}
d(x+j\omega) > nD-n|D|\ge -2n|D|.\]
The same estimate holds if
$\|M_n^u(x,E_0,\omega)\|\leq \|M_n^u(x,E,\omega)\|$. So
\begin{equation}\left |\log\|M_n^u(x,E_0,\omega)\|-\log \|M_n^u(x,E,\omega)\| \right |\leq
n\exp\left((n-1)M_0\right)\exp(2|D|n)|E_0-E|,\end{equation}when
$x\not\in\mathbb{B}_D$.
Set $r_E(n)=\frac{L(E_0,\omega_0)}{200}\left[n\exp\left((n-1)M_0\right)\exp(2|D|n)\right]^{-1}$. Therefore, if
$|E-E_0|\leq r_E(n)$, then
 \be\nonumber\left |\int_{\mathbb{T}\backslash
\mathbb{B}_D}u_n^u(x,E_0,\omega_0)-\int_{\mathbb{T}\backslash
\mathbb{B}_D}u_n^u(x,E,\omega_0) \right
|<\frac{L(E_0,\omega_0)}{200}.\ee  By the Cauchy-Schwartz inequality, it has for any $E $, $\omega$ and $n$,
\[\left|\int_{\mathbb{B}_D}u_n^u(x,E,\omega)dx\right|\leq \|u_n^u(\cdot,E,\omega)\|_{L^2(\mathbb{T})}\left(\mes\left(\mathbb{B}_D\right)\right)^{\frac{1}{2}}<\sup_{E\in\mathcal{E}}\|M^u(\cdot,E,\omega)\|_{L^2(\mathbb{T})}\exp(-\frac{\bar c_{v,a}}{2}|D|n).\]
Thus, there exists $ n_3= n_3(\lambda_a a,\lambda_v v, L(E_0,\omega_0))$ satisfying
\begin{equation}\label{n3}
 n_3= \frac{2}{\bar c_{v,a}|D|}\log\left[\frac{L(E_0,\omega_0)}{200}\sup_{E\in\mathcal{E}}\|M^u(\cdot,E,\omega_0)\|_{L^2(\mathbb{T})}^{-1}\right]
\end{equation} such that for any $n\ge  n_3$, if $|E-E_0|<r_E(n)$, then
\[|L_n(E_0,\omega_0)-L_n(E,\omega_0)|\leq\frac{L(E_0,\omega_0)}{100}.\]
\end{proof}

Now, we can get an interval centered at $E_0$, where the Lyapunov exponent is always positive.
\begin{lemma}\label{62006}Assume $L(E_0,\omega_0)>0$, $\beta(\omega_0)<c_{v,a}\min\left(\frac{L(E_0,\omega_0)}{15},|D|\right)$ and $\lambda_v>\lambda_0$. There exists $r_E>0$ such
that for  any $|E-E_0|<r_E$,
 \[
\frac{6}{5}L(E_0,\omega_0)>L(E,\omega_0)>\frac{4}{5}L(E_0,\omega_0).\]
\end{lemma}
\begin{proof}
By the subadditive property, there exists
$n_4=n_4(\lambda_aa,\lambda_vv,\omega_0,E_0)$ such that for any $n\ge n_4$,
$$L_n(E_0,\omega_0)-L(E_0,\omega_0)<\frac{L(E_0,\omega_0)}{100}\ \mbox{and\ } L_{n}(E_0,\omega_0)-L_{2n}(E_0,\omega_0)<\frac{L(E_0,\omega_0)}{100}.$$
Thus, for any $|E-E_0|\le r_E(2n_4)$, it has
\begin{eqnarray}\nonumber L_{n}(E,\omega_0)&\geq& L(E_0,\omega_0)-
|L_{n}(E,\omega_0)-L_{n}(E_0,\omega_0)|-|L_{n}(E_0,\omega_0)-L(E_0,\omega_0)|\\&>&
L(E_0,\omega_0)-\frac{L(E_0,\omega_0)}{100}-\frac{L(E_0,\omega_0)}{100}=\frac{49}{50}L(E_0,\omega_0)>0,\nonumber
\end{eqnarray}
and
\begin{eqnarray*}|L_{n}(E,\omega_0)-L_{2n}(E,\omega_0)|&\leq&
|L_{n}(E,\omega_0)-L_{n}(E_0,\omega_0)|+|L_{n}(E_0,\omega_0)-L_{2n}(E_0,\omega_0)|+|L_{2n}(E_0,\omega_0)-L_{2n}(E,\omega_0)|\\
&<&\frac{L(E_0,\omega_0)}{100}+\frac{L(E_0,\omega_0)}{100}+\frac{L(E_0,\omega_0)}{100}=\frac{3}{100}L(E_0,\omega_0)<\frac{1}{10}L_{n}(E,\omega_0).
\nonumber\end{eqnarray*}
Thus, Lemma \ref{62002}
applies for $L_{n}(E,\omega_0)$. It implies that if $n>n_2$, then
\[L(E,\omega_0)>L_n(E,\omega_0)-|L_n(E,\omega_0)-L_{2n}(E,\omega_0)|-\exp(-\bar c_u\delta^2 n)>\frac{441}{500}L(E_0)-\exp(-\bar c_u\delta^2n)>\frac{4}{5}L(E_0,\omega_0).\]
Let $r_E=r_E(\max(n_2,n_3,2n_4))$ and $L(E,\omega_0)<\frac{6}{5}L(E_0,\omega_0)$ by similar computations.
\end{proof}
\begin{remark}\label{t3}
  Due to the compactness in $E$ and the joint continuity of $L(E,\omega)$, there exists $r_{\omega}$ such that for any $|\omega-\omega_0|\le r_{\omega}$ and $|E-E_0|\le r_E$,
  \[
\frac{5}{4}L(E_0,\omega_0)>L(E,\omega)>\frac{3}{4}L(E_0,\omega_0).\]
\end{remark}

When we consider the Schr\"ondinger operators, we can calculate the expression of $r_{\omega}$:
\begin{lemma}\label{disomega}
  Assume $L^s(E_0,\omega_0)>0$, $\beta(\omega_0)<\frac{c_s}{100}L^s(E_0,\omega_0)$ and $\lambda_s>\lambda_0^s$. There exist $r^s_E=r^s_E(\lambda_s,v,E_0,\omega_0)$ and $r^s_{\omega}=r^s_{\omega}(\lambda_s,v,E_0,\omega_0)$ such that for any $|\omega-\omega_0|<r^s_{\omega}$, $|E-E_0|<r^s_E$ and $\beta(\omega)<\frac{c_s}{100}L^s(E_0,\omega_0)$,
  \[
\frac{6}{5}L^s(E_0,\omega_0)>L^s(E,\omega)>\frac{4}{5}L^s(E_0,\omega_0).\]
\end{lemma}
\begin{proof}
\begin{eqnarray}
&&\bigg |\|M^s_n(x,E,\omega_0)\|-\|M^s_n(x,E,\omega)\|\bigg |
\leq \|M^s_n(x,E,\omega_0)-M^s_n(x,E,\omega)\|\\
&\leq&\sum_{j=0}^{n-1}\bigg(\|M^s(x+(n-1)\omega_0,E,\omega_0)\times\cdots\times
M^s(x+(j+1)\omega_0,E,\omega_0)\|\times \nonumber\\&\ &\ \ \ \
\|M^s(x+j\omega_0),E,\omega_0)-
M^s(x+j\omega,E,\omega)\|\times\|M^s(x+(j-1)\omega,E,\omega)\times\cdots\times
M^s(x,E,\omega)\|\bigg)
\nonumber\\
&\leq&n^2\lambda_s V\exp\left((n-1)M_0^s\right)|\omega_0-\omega|\nonumber,
\end{eqnarray}where $V=\max_{\mathbb{T}}(v'(x))$.
Like the proof of Lemma  \ref{62003}, similar computations  show that for any $n\ge n^s_4$, it has
\[|L^s_n(E_0,\omega)-L^s_n(E_0,\omega_0)|\leq\frac{L^s(E_0,\omega_0)}{100},\]
when $|\omega-\omega_0|\le r^s_{\omega}(n)$, where $r^s_{\omega}(n)=\frac{L^s(E_0,\omega_0)}{400n\lambda_s V}\exp(-nM_0^s)$. Similarly,
\[|L^s_n(E_0,\omega_0)-L^s_n(E,\omega_0)|\leq\frac{L^s(E_0,\omega_0)}{100},\]
when $|E-E_0|\le r_E^s(n)$, where $r^s_E=\frac{L^s(E_0,\omega_0)}{200n}\exp(-nM_0^s)$.
Combining them with (\ref{disE}), we have
\[|L^s_n(E,\omega)-L^s_n(E_0,\omega_0)|\leq\frac{L^s(E_0,\omega_0)}{50},\]
when $|\omega-\omega_0|\le r_{\omega}(n)$ and $|E-E_0|\le r_E(n)$. Thus, Lemma \ref{62002} holds for $L^s_n(E,\omega)$ and this lemma is proved similarly as Lemma \ref{62006} with the settings $r^s_{\omega}=r^s_{\omega}(\max(n^s_2,2n^s_4))$ and $r^s_E=r^s_E(\max(n^s_2,2n^s_4))$, where
\begin{equation}
  \label{n2s}n_2^s:=\frac{60M_0^s}{\bar c_s\min^2\left(1,\frac{L^s(E_0,\omega_0)}{100}\right)}
\end{equation}
and $n_4^s$ is the integer which makes $L_n^s(E_0,\omega_0)-L^s(E_0,\omega_0)<\frac{L^s(E_0,\omega_0)}{100}$ for any $n\ge n_4^s$.
\end{proof}

\section{Proofs of the rest Theorems }
Before showing the proofs,  we first need  the following Lemma (Theorem 1.5 in \cite{AJS}) to get  the  uniform convergence of $u^u_n(x,E,\omega)$:
\begin{lemma}\label{contin}
The functions $\mathbb{R}\times C^{\omega}(\mathbb{T},\mathcal{L}(\mathbb{C}^d,\mathbb{C}^d))\ni (\omega,A)\mapsto L_k(\omega,A)\in[-\infty,\infty)$ are continuous at any $(\omega',A')$ with $\omega'\in \mathbb{R}\setminus \mathbb{Q}$. Here $C^{\omega}(\mathbb{T},\mathcal{L}(\mathbb{C}^d,\mathbb{C}^d))$ means the set of the functions which are complex analytic from $\mathbb{T}$ to $\mathcal{L}(\mathbb{C}^d,\mathbb{C}^d)$.
\end{lemma}

\begin{lemma}\label{uuupblem}
  Assume $L(E_0,\omega_0)>0$, $\beta(\omega_0)<c_{v,a}\min\left(\frac{L(E_0,\omega_0)}{15},|D|\right)$ and $\lambda_v>\lambda_0$. Then there exists $n_5=n_5(\omega_0, E_0,r_E,r_{\omega},$
  $\lambda_a,a,\lambda_v,v)$ such that for any $n>n_5$, $x\in \mathbb{T}$, $E\in [E_0- r_E,E_0+r_E]$ and irrational $\omega\in [\omega_0-r_{\omega},\omega_0+r_{\omega}]$,
  \begin{equation}\label{uuupb}
  u_n^u(x,E,\omega)\leq \frac{6}{5}L(E,\omega).
  \end{equation}
\end{lemma}
\begin{proof}
  Furman \cite{F} proved the uniformity   in x for any continuous
cocycle on a uniquely ergodic system. Then, due to the continuity
of $L^a$ in $E$ by Lemma \ref{contin} and the compactness, we have
\begin{equation}\nonumber
  \limsup\limits_{n\to\infty}u^a_n(x,E,\omega)\leq L^a(E,\omega).
\end{equation}uniformly in $x\in\mathbb{T}$ and $E\in [E_0-r_E,E_0+r_E]$. Similarly,
  \begin{equation}\nonumber
    \limsup\limits_{n\to\infty}\frac{1}{2n}\sum_{j=0}^{n-1}d(z+j\omega,\omega)\leq D
  \end{equation}
  uniformly in $x\in\mathbb{T}$. Thus, this lemma follows directly from (\ref{re1}).
\end{proof}
\begin{remark}\label{lar-lam1}
By Theorem \ref{posile}, if $\lambda_v>\lambda_p$, then $L(E,\omega)$ is always positive. Thus, (\ref{uuupb})  holds for any $E\in\mathcal{E}$ and  irrational $\omega$.
\end{remark}
\begin{remark}\label{schlar}
  For the Schr\"odinger operators, due to the joint continuity of Lyapunov exponent in $(E,\omega)$, we have that if $L^s(E_0,\omega_0)>0$, $\beta(\omega_0)<\frac{c_s}{15}L^s(E_0,\omega_0)$ and $\lambda_s>\lambda_0^s$, then for any $n>n^s_5$, $x\in \mathbb{T}$, $|E-E_0|\leq r^s_E$ and $|\omega-\omega_0|\leq r^s_{\omega}$, it has
  \begin{equation}\label{uupb}
  u^s_n(x,E,\omega)\leq \frac{6}{5}L^s(E,\omega).
  \end{equation}What's more, by Remark \ref{schpole}, if $\lambda>\lambda_p^s(\frac{1}{50})$, then for any $n\ge 1$, $x\in\mathbb{T}$, $E\in\mathcal{E}_s$ and  irrational $\omega$,
  \begin{equation}
    \label{uupblam}
    u^s_n(x,E,\omega)<\frac{51}{50}\log\lambda\ \mbox{and}\ L_{2n}^s(E,\omega)>\frac{9}{10}L_n^s(E,\omega).
  \end{equation}
\end{remark}~\\

\begin{proof}[Proof of Theorem \ref{sldt}]Remark \ref{t3} proves the property that $L(E,\omega)$ is positive in the neighborhood of $(E_0,\omega_0)$. Now we start the proofs of the  sharp LDTs for $u_n^u(x,E,\omega)$ and $u_n(x,E,\omega)$.

Because
\[M^u_n(x+k\omega,E,\omega)M^u_k(x,E,\omega)=M^u_k(x+n\omega,E,\omega)M^u_n(x,E,\omega)\]
and $\|A^{-1}\|=\|A\|\geq 1$ if $\det A=1$, so
\[\left
|\log\|M_n^u(x+k\omega,E,\omega)\|-\log\|M_n^u(x,E,\omega)\|\right|\leq
\log\|M^u_k(x,E,\omega)\|+\log\|M^u_k(x+n\omega,E,\omega)\|.\]
Due to (\ref{uuupb}), if $k^2>\frac{2n^2_5(M_0-D)}{L(E_0,\omega_0)}\ge n_5^2$, $|E-E_0|\leq r_E$ and $|\omega-\omega_0|\leq r_{\omega}$, then
\[\left|u_n^u(x+k\omega,E,\omega)-u_n^u(x,E,\omega)\right|\leq \frac{12k}{5n}L(E,\omega).\]
It implies that
\begin{eqnarray}\label{21124}\left|u^u_n(x,E,\omega)-\frac{1}{k}\sum_{j=1}^ku_n^u(x+j\omega,E,\omega)\right|&\leq&
\frac{1}{k}\left (\sum_{j=1}^{n_5}+\sum_{j=n_5+1}^k \right )[u_n^u(x,E,\omega)-u_n^u(x+j\omega,E,\omega)]\nonumber\\
&\leq &\frac{1}{nk}\left [\sum_{j=1}^{n_5}2jM_0+\sum_{j=n_5+1}^k   \frac{12}{5}jL(E,\omega)\right ]\nonumber\\
&& \ \ \ -\frac{1}{2nk}\left[\sum_{j=1}^{n_5}\left(\sum_{m=0}^{j-1}d(x+m\omega)-2jD\right)+\sum_{j=1}^{n_5}\left(\sum_{m=0}^{j-1}d(x+(m+n)\omega)-2jD\right)\right]
.\nonumber
\end{eqnarray}If
\[\sum_{j=1}^{n_5}\left(\sum_{m=0}^{j-1}d(x+m\omega)-2jD\right)<-\frac{nk}{4}\kappa L(E,\omega),\]
then there exists $1\leq j\leq n_5$ such that
\[\sum_{m=0}^{j-1}d(x+m\omega)-2jD<-\frac{n}{4}\kappa L(E,\omega).\]
By Lemma \ref{ldtd}, we have that if $\beta(\omega)< \frac{c_{a,v}}{4}\kappa L(E,\omega)$, then
\[\mes \left(\left\{x\in\mathbb{T}:\sum_{m=0}^{j-1}d(x+m\omega)-2jD<-\frac{n}{4}\kappa L(E,\omega)\right\}\right)\leq \exp\left(-\bar c_{v,a}\frac{n}{4j}\kappa L(E,\omega)j\right)=\exp\left(-\frac{\bar c_{v,a}}{4}\kappa L(E,\omega)n\right).\]
Thus, there exist
\begin{equation}
  \label{n6}n_6:=\frac{16\log\frac{32}{\bar c_{v,a}\kappa L(E_0,\omega_0)}}{\bar c_{v,a}\kappa L(E_0,\omega_0)}
\end{equation} and  $\mathbb{B}$ satisfying
\[\mes(\mathbb{B})\leq 2n_5 \exp\left(-\frac{\bar c_{v,a}}{4}\kappa L(E,\omega)n\right)\leq \exp\left(-\frac{\bar c_{v,a}}{8}\kappa L(E,\omega)n\right),\]
such that if $x\not\in \mathbb{B}$ and $n=\frac{4k}{\kappa} >n_6$, then
\[-\frac{1}{2nk}\left[\sum_{j=1}^{n_5}\left(\sum_{m=0}^{j-1}d(x+m\omega)-2jD\right)+\sum_{j=1}^{n_5}\left(\sum_{m=0}^{j-1}d(x+(m+n)\omega)-2jD\right)\right]
\leq \frac{1}{4}\kappa L(E,\omega)\]
and
\begin{eqnarray*}\left|u^u_n(x,E,\omega)-\frac{1}{k}\sum_{j=1}^ku_n^u(x+j\omega,E,\omega)\right|
&\leq &\frac{1}{nk}\left [\sum_{j=1}^{n_5}2j(M_0-D)+\sum_{j=n_5+1}^k   \frac{12}{5}jL(E,\omega)\right ]+\frac{1}{4}\kappa L(E,\omega)\nonumber\\
&\leq &\frac{n_5^2(M_0-D)}{nk}+\frac{6k}{5n}L(E,\omega)+\frac{1}{4}\kappa L(E,\omega)\\
&\leq &\frac{\kappa}{8}L(E_0,\omega_0)+\frac{11}{20}\kappa L(E,\omega)< \frac{3}{4}\kappa L(E,\omega).
\end{eqnarray*}
Let $\delta=\frac{\kappa}{4}L(E,\omega)$ in Remark \ref{ergu} and \begin{equation}
  \label{checkn}\check n=\max(n_1,n_2,n_3,2n_4,n_5,n_6),
\end{equation} where $n_1,\ n_2,\ n_3,\ n_4,\ n_5$ and  $n_6$ are defined   in (\ref{n1}), (\ref{n2}), (\ref{n3}), the proof of Lemma \ref{62006}, the proof of Lemma \ref{uuupblem} and (\ref{n6}), respectively. Then, redefine $r_E$ as
\[r_E=r_E(\check n)=\frac{L(E_0,\omega_0)}{200\check n}\exp\left((1-\check n)M_0-2|D|\check n\right).\]
Therefore, if $\beta<c_{v,a}\min\left(\frac{\kappa}{4} L(E,\omega),\frac{L(E_0,\omega)}{15},|D|\right)$, $\lambda_v>\lambda_0$ and $n>\check n$, then for any $|E-E_0|\leq r_E$ and  $|\omega-\omega_0|\leq r_{\omega}$,
\begin{eqnarray}\label{sldtue}
  \mes\left\{x:\left|u_n^u(x,E,\omega)-L_n(E,\omega)\right|<\kappa L(E,\omega)\right\}&<&\exp\left(-\frac{\bar c_{v,a}}{8}\kappa L(E,\omega)n\right)+\exp\left(-\frac{\bar c_{v,a}}{25}\kappa^2 L(E,\omega)n\right)\\
  &<&\exp\left(-\frac{\bar c_{v,a}}{30}\kappa^2 L(E,\omega)n\right).\nonumber
\end{eqnarray}Similarly, we can get the sharp LDT for $u_n(x,E,\omega)$ by  the following relationship between $u_n(x,E,\omega)$ and $u_n^u(x,E,\omega)$:
\[u_n(x,E,\omega)=u_n^u(x,E,\omega)+\frac{1}{n}\sum_{j=0}^{n-1}\left(\log|\lambda_a\tilde a(x+j\omega)|-\log|\lambda_aa(x+(j+1)\omega)|\right).\]
\end{proof}
\begin{remark}\label{schlarsldt}
It is easily seen that if $\lambda_v>\lambda_p$, then the sharp LDT (\ref{sldte}) holds for any $E\in\mathcal{E}$ and irrational $\omega$.
\end{remark}
\begin{remark}
  For the Schr\"odinger operators, we have similarly that if $k=\frac{\kappa}{4}n>n^s_5\left(\frac{2M^s_0}{L^s(E_0,\omega_0)}\right)^{\frac{1}{2}}$, $|E-E_0|\leq r^s_E$ and $|\omega-\omega_0|\leq r^s_{\omega}$, then
  \[\left|u_n^s(x+k\omega,E,\omega)-u_n^s(x,E,\omega)\right|\leq \frac{12k}{5n}L^s(E,\omega)\]
and
\begin{eqnarray}\label{schdis}\left|u^s_n(x,E,\omega)-\frac{1}{k}\sum_{j=1}^ku_n^s(x+j\omega,E,\omega)\right|&\leq&
\frac{1}{k}\left (\sum_{j=1}^{n_5}+\sum_{j=n_5+1}^k \right )[u_n^s(x,E,\omega)-u_n^s(x+j\omega,E,\omega)]\\
&\leq &\frac{1}{nk}\left [\sum_{j=1}^{n_5}2jM_0^s+\sum_{j=n_5+1}^k   \frac{12}{5}jL^s(E,\omega)\right ]\nonumber\\
&\leq & \frac{n_5^2M_0^s}{nk}+\frac{6k}{5n}L^s(E,\omega)\leq \frac{3}{5}\kappa L^s(E,\omega).\nonumber
\end{eqnarray}
Combining it with Remark \ref{ergs}, we have that if $\beta(\omega)<\frac{1}{5}c_s\kappa L^s(E,\omega)$ and $\lambda>2\epsilon_0^{-1}$, then
\begin{equation}
  \label{schsldte}
  \mes\left\{x:\left|u_n^s(x,E,\omega)-L^s_n(E,\omega)\right|<\kappa L^s(E,\omega)\right\}<\exp(-\frac{\bar c_s}{10}\kappa^2L^s(E,\omega)).
\end{equation}
\end{remark}~\\

\begin{proof}[Proof of Theorem \ref{holder}]
Similar to the proofs of Lemma \ref{62001} and Lemma \ref{62002},  by the sharp LDT (\ref{sldtue}) with $\kappa=\frac{1}{20}$, the Avalanche Principle can be applied again, and we have
\begin{equation}\label{apforl}
 |L(E,\omega)+L_n(E,\omega)-2L_{2n}(E,\omega)|<\exp(-\frac{1}{48000}\bar c_{v,a}L(E,\omega)n)\leq \exp(-10^{-5}\bar c_{v,a}L(E_0,\omega_0)n).
\end{equation}
On the other hand, by (\ref{disE}), (\ref{60002}) and Lemma \ref{uuupb}, it implies that
\begin{eqnarray*}&&\left |\log\|M_n^u(x,E_1,\omega)\|-\log \|M_n^u(x,E_2,\omega)\| \right |\leq
 \frac{\left|\|M_n^a(x,E_1,\omega)\|-\|M_n^a(x,E_2,\omega)\|\right|}{\prod_{j=0}^{n-1}|d(x+j\omega)|^{\frac{1}{2}}}\\
 &\leq&\sum_{j=0}^{n-1}\left(\|\prod_{m=1}^{n-j}M^a(x+(n-m)\omega,E_1,\omega)\|\times
\|\prod_{m=j-1}^0M^a(x+m\omega,E_2,\omega)\|\right)\times \exp(-nD)\nonumber\\
&&\times\exp\left (-\sum_{j=0}^{n-1}\frac{1}{2}b(x+j\omega)+nD\right )\times|E_1-E_2|\nonumber\\
&=&\left (\sum_{j=1}^{n_5}+\sum_{j=n_5+1}^{n-n_5}
 +\sum_{j=n-n_5+1}^{n}\right ) \left(\|\prod_{m=1}^{n-j}M^a(x+(n-m)\omega,E_1,\omega)\|\times
\|\prod_{m=j-1}^0M^a(x+m\omega,E_2,\omega)\|\right)\times \exp(-nD)\nonumber\\
&&\times\exp\left (-\sum_{j=0}^{n-1}\frac{1}{2}b(x+j\omega)+nD\right )\times|E-E_0|\nonumber\\
&\leq &\left\{2\sum_{j=1}^{n_5}\exp\left((M_0-D)n_5+\frac{6}{5}\max_{i=1,2}\left(L(E_i,\omega)\right)n\right)+\sum_{j=n_5+1}^{n-n_5}\exp\left(\frac{6}{5}\max_{i=1,2}\left(L(E_i,\omega)\right)n\right)\right\}\nonumber\\
&&\times\exp\left (-\sum_{j=0}^{n-1}\frac{1}{2}b(x+j\omega)+nD\right )\times|E_1-E_2|\nonumber\\
&\leq &\sum_{j=0}^{n-1}\exp\left(\frac{3}{2}L(E_0,\omega_0)n\right)\times\exp\left (-\sum_{j=0}^{n-1}\frac{1}{2}b(x+j\omega)+nD\right )\times|E_1-E_2|\nonumber.
\end{eqnarray*}
Due to Lemma \ref{ldtd}, we have that there exists $\mathbb{B}'$ satisfying $\mes \mathbb{B}'<\exp\left(-\bar c_{v,a}L(E_0,\omega_0)n\right)$ such that if $x\not\in\mathbb{B}'$, then
\[\left |\log\|M_n^u(x,E_1,\omega)\|-\log \|M_n^u(x,E_2,\omega)\| \right |\leq n\exp\left(2L(E_0,\omega_0)n\right)|E_1-E_2|.\]
Thus
\begin{eqnarray*}
|L_n(E_1,\omega)-L_n(E_2,\omega)|&=&
\int_{\mathbb{T}\backslash\mathbb{B}'}|u_n^u(x,E_1,\omega)-u_n^u(x,E_2,\omega)|dx+\int_{\mathbb{B}'}|u_n^u(x,E_1,\omega)-u_n^u(x,E_2,\omega)|dx\\
&<& \exp
(2L(E_0,\omega_0)n)|E_1-E_2|+2\sup\limits_{E\in\mathcal{E}}\|M^u(\cdot,E,\omega)\|_{L^2(\Omega)}\exp(-\frac{\bar c_{v,a}}{2}L(E_0,\omega_0)n)\nonumber.
\end{eqnarray*}
Combining it with (\ref{apforl}), we have
\begin{eqnarray*}
  |L(E_1,\omega)-L(E_2,\omega)|&\leq &
|L(E_1,\omega)+L_n(E_1,\omega)-2L_{2n}(E_1,\omega)|+|L(E_2,\omega)+L_n(E_2,\omega)-2L_{2n}(E_2,\omega)|\\
&\ &\ \ +|L_n(E_1,\omega)-L_n(E_2,\omega)|+2|L_{2n}(E_1,\omega)-L_{2n}(E_2,\omega)|\nonumber\\
&< &2\exp(-10^{-5}\bar c_{v,a}L(E_0,\omega_0)n)+2\exp
(4L(E_0,\omega_0)n)|E_1-E_2|\nonumber\\
&\ &\ \ +4\sup\limits_{E\in\mathcal{E}}\|M^u(\cdot,E,\omega)\|_{L^2(\Omega)}\exp(-\frac{\bar c_{v,a}}{2}L(E_0,\omega_0)n)\nonumber\\
&< &3\exp(-10^{-5}\bar c_{v,a}L(E_0,\omega_0))+2\exp
(4L(E_0,\omega_0)n)|E_1-E_2|.
\end{eqnarray*}
Note that $M_0-D\ge L(E_0,\omega_0)$. Thus, if $|E_1-E_2|<2r_E$, then there exists an integral $n>\check n$ such that
\[\exp\left(-(10^{-5}\bar c_{v,a}+4)L(E_0,\omega_0)(n+1)\right)\leq |E_1-E_2|\leq  \exp\left(-(10^{-5}\bar c_{v,a}+4)L(E_0,\omega_0)n\right).\]
Then
\[
 |L(E,\omega_0)-L(E_0,\omega_0)|\le 5\exp(-10^{-5}\bar c_{v,a} L(E_0,\omega_0)n)<\exp(-(2\times 10^5)^{-1}\bar c_{v,a}L(E,\omega)(n+1))\leq |E-E_0|^{\tau},
\]where $\tau=\frac{\bar c_{v,a}}{2\bar c_{v,a}+8\times 10^5}$.
\end{proof}
\begin{remark}Let us outline the proof of Remark \ref{schholder}. If $n>\frac{10M_0^s}{L^s(E_0,\omega_0)}n^s_5$, then
  \begin{eqnarray}
    &&\left|\log\|M_n^s(x,E,\omega_1)\|-\log\|M_n^s(x,E,\omega_2)\right|\leq \|M_n^s(x,E,\omega_1)-M_n^s(x,E,\omega_2)\|\\
    &\leq &\left (\sum_{j=1}^{n_5}+\sum_{j=n_5+1}^{n-n_5}
 +\sum_{j=n-n_5+1}^{n}\right ) \left(\|\prod_{m=1}^{n-j}M^s(x+(n-m)\omega_1,E,\omega_1)\|\times
\|\prod_{m=j-1}^0M^s(x+m\omega_2,E,\omega_2)\|\times \left(j\lambda V|\omega_1-\omega_2|\right)\right) \nonumber\\
 &\leq &\left\{2\sum_{j=1}^{n_5}\exp\left(n_5M_0^s+\frac{6}{5}\max_{i=1,2}\left(L^s(E,\omega_i)\right)(n-n_5)\right)+\sum_{j=n_5+1}^{n-n_5}\exp\left(\frac{6}{5}\max_{i=1,2}\left(L^s(E,\omega_i)\right)n\right)\right\}\times \lambda_s nV|\omega_1-\omega_2| \nonumber\\
 &\leq&\sum_{j=1}^n\exp\left(\frac{7}{5}L^s(E_0,\omega_0)n\right)\times\lambda_s nV|\omega_1-\omega_2|.\nonumber
  \end{eqnarray}Thus, there exists
  \begin{equation}
    \label{n7}n^s_7:=-\frac{5\log\frac{L^s(E_0,\omega_0)}{5\lambda_s V}}{L^s(E_0,\omega_0)}
  \end{equation}such that for any $n\ge n^s_7$,
  \[\left|L^s_n(E,\omega_1)-L^s_n(E,\omega_2)\right|\leq \lambda nV\exp\left(\frac{7}{5}L^s(E_0,\omega_0)n\right)\times|\omega_1-\omega_2|<\exp(2L^s(E_0,\omega_0)n)\times|\omega_1-\omega_2|.\]
  Similarly, for any $n>\frac{10M_0^s}{L^s(E_0,\omega_0)}n^s_5$,
   \[\left|L^s_n(E_1,\omega)-L^s_n(E_2,\omega)\right|\leq \exp\left(\frac{7}{5}L^s(E_0,\omega_0)n\right)\times|E_1-E_2|.\]
  On the other hand, applying (\ref{schsldte}) with $\kappa=\frac{1}{20}$ to the Avalanche Principle, we have
  \[\left|L^s(E,\omega)+L_n^s(E,\omega)-2L_{2n}^s(E,\omega)\right|\leq \exp\left(-10^{-5}\bar c_s L^s(E_0,\omega_0)n\right).\]
  Let
  \begin{equation}
    \label{checkns}
    \check n_{s}:=\max\left(n_2^s,2n_4^s,80\left(\frac{2M_0^s}{L^s(E_0,\omega_0)}\right)^{\frac{1}{2}}n_5^s,\frac{10M_0^s}{L^s(E_0,\omega_0)}n_5^s,n_7^s\right),
  \end{equation}
  where $n_2^s$ is defined in (\ref{n2s}), $n_4^s$ in Lemma \ref{disomega}, $n_5^s$ in Remark \ref{schlar} and $n_7^s$ in (\ref{n7}). Then, we redefine  $r_E^s$ and $r_{\omega}^s$ as follows:
  \[r^s_E=\frac{L^s(E_0,\omega_0)}{200\check n_{s}}\exp(-5M_0^s\check n_{s}),\ r^s_{\omega}=\frac{L^s(E_0,\omega_0)}{400\max_{\mathbb{T}}(v'(x))\check n_{s}^2}\exp(-5M_0^s\check n_{s}).\]
  Note that $5M_0^s>(10^{-5}\bar c_v+4)L^s(E_0,\omega_0)$.
  Thus, for any $E_1,E_2\in [E_0-r_E^s,E_0+r_E^s]$, $\omega_1,\omega_2\in[\omega_0-r_{\omega}^s,\omega_0+r_{\omega}^s]$ satisfying $\max(\beta(\omega_1),\beta(\omega_2))<\frac{1}{15}L^s(E_0,\omega_0)$,  there exist $n^s_E$ and $n^s_{\omega}$ such that
  \[|E_1-E_2|\sim \exp\left(-(10^{-5}\bar c_v+4)L^s(E_0,\omega_0)n_E\right),\ |\omega_1-\omega_2|\sim \exp\left(-(10^{-5}\bar c_v+4)L^s(E_0,\omega_0)n_{\omega}\right).\]
  Therefore,
  \begin{eqnarray*}
  |L^s(E_1,\omega_1)-L^s(E_1,\omega_2)|&\leq &
|L^s(E_1,\omega_1)+L^s_{n^s_{\omega}}(E_1,\omega_1)-2L^s_{2n^s_{\omega}}(E_1,\omega_1)|+|L^s(E_1,\omega_2)+L^s_{n^s_{\omega}}(E_1,\omega_2)-2L^s_{2n^s_{\omega}}(E_1,\omega_2)|\\
&\ &\ \ +|L^s_{n^s_{\omega}}(E_1,\omega_1)-L^s_{n^s_{\omega}}(E_1,\omega_2)|+2|L^s_{2n^s_{\omega}}(E_1,\omega_1)-L^s_{2n^s_{\omega}}(E_1,\omega_2)|\nonumber\\
&< &2\exp(-10^{-5}\bar c_vL(E_0,\omega_0)n^s_{\omega})+2\exp
\left(4L(E_0,\omega_0)n^s_{\omega}\right)|\omega_1-\omega_2|\nonumber\\
&< &4\exp(-10^{-5}\bar c_{v,a}L(E_0,\omega_0)n^s_{\omega})<|\omega_1-\omega_2|^{\tau_s},
\end{eqnarray*}
and similarly,
$$|L^s(E_1,\omega_2)-L^s(E_2,\omega_2)|<|E_1-E_2|^{\tau_s}.$$
\end{remark}
\begin{remark}If $\lambda_s>\max\left(\lambda_p^s\left(\frac{1}{50}\right),5V\right)$, then $L^s(E,\omega)$ is always positive for any $E\in\mathcal{E}_s$ and irrational $\omega$. Thus, we do not need to apply the LDT and the Avalanche Principle to obtain the interval where the Lyapunov exponent is positive. Furthermore, due to Remark \ref{schlar}, we have
 \[n_5^s=1,\ 80\left(\frac{2M_0^s}{L^s(E_0,\omega_0)}\right)^{\frac{1}{2}}n_5^s<160,\ \frac{10M_0^s}{L^s(E_0,\omega_0)}n_5^s<20,\ n_7^s<20.\]
Overall, the integers $n^s_{E}$ and $n^s_{\omega}$ only need to be larger than $160$. Therefore, (\ref{schlarholder}) holds for any $|E_1-E_2|<\lambda_s^{-800}$ and $|\omega_1-\omega_2|<\lambda_s^{-800}$ satisfying $\max\left(\beta(\omega_1),\beta(\omega_2)\right)<\frac{c_s}{16}\log\lambda_s$.
\end{remark}

\end{document}